\documentclass[conference,onecolumn]{IEEEtran}

\usepackage{graphicx}
\usepackage[driverfallback=dvipdfm]{hyperref} 
\usepackage{epstopdf}
\usepackage{mathrsfs}
\usepackage{mathtools,xparse}
\usepackage{amssymb}
\usepackage{amsmath}
\usepackage{bbm}
\usepackage{amsthm}
\usepackage{bm}
\usepackage{algorithm}
\usepackage{algorithmic}
\makeatletter
\def\blfootnote{\gdef\@thefnmark{}\@footnotetext}
\makeatother
\DeclarePairedDelimiter\abs{\lvert}{\rvert}
\DeclarePairedDelimiter{\norm}{\lVert}{\rVert}
\DeclarePairedDelimiter{\ceil}{\lceil}{\rceil}
\DeclarePairedDelimiter{\bra}{\lbrack}{\rbrack}
\DeclarePairedDelimiter{\bre}{\lbrace}{\rbrace}
\DeclarePairedDelimiter{\para}{(}{)}

\newtheorem{theorem}{Theorem}
\theoremstyle{definition}
\newtheorem{definition}{Definition}
\newtheorem{lemma}{Lemma}

\newtheorem{claim}{Claim}
\newtheorem{proposition}{Proposition}

\IEEEoverridecommandlockouts
\begin{document}
\title{Sequential Controlled Sensing for Composite Multihypothesis Testing}
\author{\IEEEauthorblockN{Aditya Deshmukh}
\and
\IEEEauthorblockN{Srikrishna Bhashyam}
\and
\IEEEauthorblockN{Venugopal V. Veeravalli}
\thanks{A part of this work was presented at the 2018 Asilomar Conference on Signals, Systems, and Computers under the title `Controlled Sensing for Composite Multihypothesis Testing with Application to Anomaly Detection' \cite{deshmukh}.}}
\maketitle
\IEEEpeerreviewmaketitle
\begin{abstract}
The problem of multi-hypothesis testing with controlled sensing of observations is considered. The distribution of observations collected under each control is assumed to follow a single-parameter exponential family distribution. The goal is to design a policy to find the true hypothesis with minimum expected delay while ensuring that probability of error is below a given constraint. The decision maker can control the delay by intelligently choosing the control for observation collection in each time slot. We derive a policy that satisfies the given constraint on the error probability. We also show that the policy is asymptotically optimal in the sense that it asymptotically achieves an information-theoretic lower bound on the expected delay.
\end{abstract}
\section{Introduction}
Sequential controlled sensing is a stochastic framework wherein a decision-maker collects observations from a set of controls by sequentially choosing a control and obtaining an observation associated with that control. This paradigm is encountered in information-gathering systems with multiple degrees of freedom that can be controlled adaptively to achieve a given statistical inference task. In traditional control systems, the control is responsible for governing the state of the system. On the other hand, in controlled sensing, the control governs the quality of observations.

Some applications of controlled sensing are target detection, tracking, classification and dynamic sensor selection. A widely studied problem that can be considered as a special case of controlled sensing is that of anomaly detection. Some applications of anomaly detection include identification of defective batches in manufacturing, detection of abnormal behaviour of machines, and outlier detection in datasets. Another problem studied by the computer science community, which can also be considered to be a special case of controlled sensing, is best arm identification in multi-armed bandits. Controlled sensing has potential applications in diagnostic inference \cite{horvitz1988decision}, particularly clinical decision support systems, which help clinicians in taking diagnostic decisions. Taking measurements from medical sensors can be expensive and so a potential inference problem would be to find a sequential policy to minimize the number of measurements taken to find the correct hypothesis related to a patient's state of health, with high probability.

 We consider the problem of finding the true hypothesis from a finite set of composite hypotheses, with minimum expected delay in a sequential controlled sensing setting, while ensuring that a constraint on the probability of error is satisfied. To achieve this goal, the decision-maker has to intelligently choose a control at each time step in order to make best use of the observations, decide when to stop, and find an appropriate estimate of the true hypothesis.
\subsection{Related Work}
Chernoff pioneered controlled sensing in his seminal work \cite{chernoff}. Chernoff considered the problem of composite binary hypothesis testing in a sequential controlled sensing setting. He assumed that the distributions under both hypotheses were parametrized and the two sets of parameters under the hypotheses were disjoint and finite. The set of controls was assumed to be finite as well. Chernoff proposed a policy, known as `Procedure A', and proved that it is asymptotically optimal under certain positivity constraints on Kullback-Leibler divergences. Albert \cite{albert} extended Chernoff's results to the case where the parameter space is infinite with certain restrictions. Bessler \cite{bessler} also generalized Chernoff's work to multiple hypothesis and an infinite set of controls, but with a finite parameter space. In these papers, the authors named the control sensing problem as `sequential design of experiments'.

Nitinawarat et al. \cite{vvv} studied the problem of controlled sensing for multihypothesis testing in a setting where the distributions were assumed known, and provided an asymptotically optimal policy without the positivity assumption of prior work, and with strict guarantees on the probabilities (risks) of choosing the hypotheses incorrectly. Naghshvar et al. \cite{naghshvar2013active} considered controlled sensing for sequential multihypothesis in the non-asymptotic regime and analyzed a dynamic programming solution to find the structure of the optimal test, and also studied the problem where the number of hypotheses goes to infinity. The authors of \cite{naghshvar2013active} term the controlled sensing problem as `active sequential hypothesis testing'.

We now discuss related work in anomaly detection, which is a special case of controlled sensing. Li et al. \cite{yunli} studied outlier hypothesis testing in a setting where there is no control and all processes (taking values in finite sets) are sampled together, and provided a universally exponentially consistent policy when both anomalous and non-anomalous distributions are unknown. Cohen et al. \cite{cohen} considered the problem of anomaly detection with control when both anomalous and non-anomalous distributions are known, and provided an asymptotically optimal deterministic test. Vaidhiyan et al. \cite{Vaidhiyan2018} studied the problem of detecting an odd process among a group of Poisson point processes, in a setting where parameters of the odd and non-odd processes were unknown, and provided an asymptotically optimal policy. Prabhu et al. \cite{gayathri} generalized \cite{Vaidhiyan2018} to vector-exponential families and also considered switching costs.

Best arm identification in multi-armed bandits is a problem well studied by the computer science community. The framework of multi-armed bandits is similar to that of controlled sensing. Kaufmann et al. \cite{kaufmann2016complexity} studied the complexity of identifying best arms in a multi-armed bandit. Garivier et al. \cite{garivier2016optimal} provided an asymptotically optimal policy for best arm identification in multi-armed bandits where the distributions on the arms were assumed to belong to a single-parameter exponential family, and the parameters of these distributions were unknown.

\subsection{Paper Outline}
In Section \ref{sec:pm} we introduce the problem model. In Section \ref{sec:lb} we provide a lower bound on the expected delay of policies in the class of interest. In Section \ref{sec:overview}, we give an overview of results and some applications. In Section \ref{sec:pp}, we discuss a proposed policy. In Section \ref{sec:nr}, we provide some simulations and numerical results. Proofs of all results can be found in Appendices \ref{appendix:A}, \ref{appendix:B} and \ref{appendix:C}.
\section{Problem Model}\label{sec:pm}
\subsection {Single parameter exponential family}
The single parameter exponential family is a collection of probability distributions whose probability density/mass functions can be expressed as
\begin{equation}
p(y;\theta) = h(y) \exp \para*{\theta T(y) -A(\theta)},
\end{equation}
where $\theta$ is the parameter, (also known as the natural parameter) from some parameter set $\Psi\subset\mathbb{R}$, $T:\mathbb{R}\rightarrow \mathbb{R}$ represents the statistic, $A : \Theta \rightarrow \mathbb{R}$ is a convex function, known as the log-partition function. $A(\theta)$ can be expressed as
\begin{equation}
A(\theta) = \log \int_{-\infty}^{\infty} h(y)\exp(\theta T(y)) dy.
\end{equation}
The distribution can also be parametrized by the expectation parameter $\kappa$ which is the expected value of the statistic,
\begin{equation}
	\kappa=\mathbb{E}_{\theta}[T(Y)] =\dot{A}(\theta),
\end{equation}
where $\dot{f}$ is used to represent the derivative of a real-valued function $f$, that is, $\dot{f}=\frac{df}{dy}$.
It is known that $A$ is infinitely differentiable over the domain $\Psi$.\\

\noindent Let $b$ be the convex conjugate function of $A$,
\begin{equation}
    b(\kappa)=\sup\limits_{\theta}\para*{\kappa\theta-A(\theta)}.
\end{equation}
Then $\theta$ corresponding to $\kappa$ is given by
\begin{equation}
    \theta = \dot{b}(\kappa).
\end{equation}
The dual relationship between $\kappa$ and $\theta$ is given by
\begin{equation}
\kappa = \dot{A}(\theta) \text{ and } \theta=\dot{b}(\kappa).
\end{equation}
The KL-divergence between two distributions having natural parameters $\theta$ and $\theta'$ respectively is given by :
\begin{align}
   D(\theta||\theta')&\coloneqq \int\limits_{-\infty}^{\infty} f\para*{y;\theta}\log\frac{f\para*{y;\theta}}{f\para*{y;\theta'}}dy\\
	&=A(\theta')-A(\theta)-\dot{A}(\theta)(\theta'-\theta).
\end{align}
\subsection {Problem setup}
We consider a set of controls denoted by the finite set
\begin{equation}
\mathcal{U}=\bre*{1,2,3,\dots,\abs*{\mathcal{U}}}.
\end{equation}
The state of nature is denoted by a vector of parameters $\bm\theta$. In the general setting of controlled sensing, the observations under a control, say $u$, are assumed to follow a non-specific distribution with density, which we denote by $p(y;\bm\theta, u)$, with respect to some common measure $\mu$. In this work, we assume that $\bm\theta = \para*{\theta_1,\theta_2,\dots,\theta_{\abs*{\mathcal{U}}}}$ is a $\abs*{\mathcal{U}}$-dimensional vector and that the distribution of the observations under control $u$ is a member of a single-parameter exponential family with parameter as the $u$-th coordinate of $\bm\theta$, represented as $\theta_u\in \Psi_u$, where $\Psi_u\subset\mathbb{R}$. Let the domain of $\bm\theta$ be denoted as:
\begin{equation}
\Omega = \Psi_1\times\Psi_2\times\dots\Psi_{\abs*{\mathcal{U}}}.
\end{equation}

The probability density/mass function of observation $y$ under control $u$ and given parameters $\bm\theta$ is given by
\begin{equation}
p(y;\bm\theta,u)=h_u(y)\exp\bra*{\theta_u T_u(y)-A_u(\theta_u)},
\end{equation}
where $T_u$ is the statistic function and $A_u$ is the log-partition function of the exponential family associated with control $u$. Let $b_u$ be the convex conjugate function of $A_u$. The KL-divergence between between the distributions under control $u$ and $u'$, for control parameters $\bm\theta$ and $\bm\theta'$ is
\begin{align}
   D_u(\bm\theta||\bm\theta')&\coloneqq \int\limits_{-\infty}^{\infty} p\para*{y;\bm\theta,u}\log\frac{p\para*{y;\bm\theta,u}}{p\para*{y;\bm\theta',u}}d\mu(y)\\
	&=A_u(\theta'_u)-A_u(\theta_u)-\dot{A}_u (\theta_u)(\theta_u'-\theta_u).
\end{align}

The set of hypothesis is denoted by 
\begin{equation}
	\mathcal{M}=\bre*{1,2,3,\dots,M}.
\end{equation} 
Under hypothesis $m\in\mathcal{M}$, $\bm\theta\in \Theta_m$, where $\Theta_m\subset\Omega$. Let $\norm*{.}$ be a norm on $\mathbb{R}^{\abs*{\mathcal{U}}}$. We assume the following structure on the sets $\Theta_m$.
\begin{enumerate}
\item Each $\Theta_m$ is a disjoint finite union of sets, that is $\Theta_m=\bigcup\limits_{i=1}^{x_m} \Gamma^{(i)}_m$, where $x_m\in\mathbb{N}$ and $\Gamma^{(i)}_m\cap\Gamma^{(j)}_m=\phi$, $\forall i,j\in [x_m]$ such that $i\neq j$.
\item $\forall m\in\mathcal{M}, \forall i\in[x_m], \Gamma^{(i)}_m$ is convex and open in its own affine hull, denoted by $\text{aff}\para*{\Gamma^{(i)}_m}$.
\item $\forall m_1,m_2\in\mathcal{M}$ such that $m_1\neq m_2$, we have $\Gamma^{(i)}_{m_1}\cap\Gamma^{(j)}_{m_2}=\phi$, $\forall i\in [x_{m_1}]$, $\forall j\in [x_{m_2}]$. Note that this implies $\Theta_m$'s are mutually disjoint.
\item $\forall m\in\mathcal{M}, \forall i\in[x_m], \forall \bm\theta\in \Gamma^{(i)}_m$, $\Delta\para*{\bm\theta,\Gamma^{(j)}_{m'}}>0$ for any $m'\in\mathcal{M},j\in [x_{m'}]$ such that $m'\neq m$ or $j\neq i$,\label{assump:4}
where $\Delta(\bm{x},A):\Omega\to\mathbb{R}$ is the distance of $\bm{x}\in\Omega$ to the set $A\subset \Omega$ given by
\begin{equation}
\Delta\para*{\bm{x},A}\coloneqq\inf\bre*{\norm*{\bm{x}-\bm\theta'}:\bm\theta'\in A}.
\end{equation}
\end{enumerate}

We consider a sequential setting where at each time step $k=1,2,3,\dots$ the controller selects a control $U_k$ and gets an observation $Y_k$. All observations $(Y_k)_{k\geq 1}$ and all control selections $(U_k)_{k\geq 1}$ are assumed to be defined on a common probability space. Let $\mathcal{F}_n=\sigma(U_1,Y_1,U_2,Y_2,\dots,U_n,Y_n)$ be the sigma-algebra generated by the selected controls and observations up to time $n$. $\mathbb{P}_{\bm\theta}[.]$ and $\mathbb{E}_{\bm\theta}[.]$ denote the probability and expectation respectively conditioned that the vector of parameters is $\bm\theta$. A policy $\Phi=\para*{\bre*{U_n},\tau,\hat{m}}$ is then defined by:
\begin{itemize}
\item a sequence of controls $\bre*{U_n}$, where $U_n$  is $\mathcal{F}_{n-1}$ measurable,
\item a stopping rule $\tau$, which is a stopping time with respect to $U_1,Y_1,U_2,Y_2,\dots$, and
\item an $\mathcal{F}_{\tau}$-measurable decision $\hat{m}$ which is the policy's estimate of the true hypothesis.
\end{itemize}
Any such policy keeps taking observations by choosing controls based on past observations and chosen controls, until the stopping time. At the stopping time, the policy stops taking any further observations and choosing any further controls, and outputs an estimate of the true hypothesis. The goal is to design a policy to find the true hypothesis with minimum expected delay $\mathbb{E}_{\bm\theta}[\tau]$ while ensuring that probability of error is below a given constraint $\alpha$. Let
\begin{equation}
g(\bm\theta) \coloneqq m \text{ if } \bm\theta\in\Theta_m.
\end{equation}

\begin{definition}[$\bar{\alpha}$-correct policy]
Let $\alpha \in (0,1)$. A policy is called $\bar{\alpha}$-correct if $\forall \bm\theta\in\bigcup\limits_{m=1}^M \Theta_m$, $\mathbb{P}_{\bm\theta}[\tau<\infty]=1$ and $\mathbb{P}_{\bm\theta}[\hat{m}\neq g(\bm\theta) ]\leq \alpha$.
\end{definition}

For any state of nature parameters, an $\bar{\alpha}$-correct policy stops in finite time almost surely and detects the true hypothesis with probability of at-least $1-\alpha$. We contribute a policy which we show to be $\bar{\alpha}$-correct and asymptotically optimal in the sense that it achieves the aymptotic lower bound on expected delay as $\alpha\to 0$. We discuss this lower bound in the next section. 

\section{Lower bound}\label{sec:lb}
We first establish a lower bound on the expected delay of any $\bar{\alpha}$-correct policy. 
 \begin{lemma} \label{lem:lowerbd}
Let $\alpha\in (0,1)$. Then $\forall \bm\theta\in\bigcup\limits_{m=1}^M \Theta_m$, any $\bar{\alpha}$-correct policy satisfies
\begin{equation}\label{eq:lb}
	\mathbb{E}_{\bm\theta}[\tau]\geq \frac{d(\alpha||1-\alpha)}{D^*(\bm\theta)},
\end{equation} 
where $D^*(\bm\theta)$ is defined as,
\begin{equation}\label{eq:D}
D^*(\bm\theta)\coloneqq \sup_{\bm{q}\in\mathcal{P}}\inf_{\bm\theta'\in\bigcup\limits_{m=1}^M \Theta_m \setminus\Theta_{g(\bm\theta)}} \sum\limits_{u=1}^{\abs*{\mathcal{U}}} q_u D_u(\bm\theta||\bm\theta').
\end{equation}
$d(x||y)\coloneqq x\log\para*{\frac{x}{y}}+(1-x)\log\para*{\frac{1-x}{1-y}}$ represents the binary relative entropy function and the supremum is taken over $\mathcal{P}$, the set of all distributions over $\mathcal{U}$.
\end{lemma}
\begin{proof} The proof follows from Lemma 1 in \cite{kaufmann2016complexity}, which is stated for multi-armed bandit models, but can be applied to the case of sequential controlled sensing due to similarity in the paradigms.
\end{proof} 
We further analyze $D^*(\bm\theta)$ to gain insights and discover properties which might help us in designing a good policy for the problem in consideration.
\begin{proposition}\label{prop:q}
The supremum in \eqref{eq:D} is a maximum and attained $\forall \bm\theta\in\bigcup\limits_{m=1}^M \Theta_m$ at
\begin{equation}
	\bm{q}^*(\bm\theta) = \mathop{\arg\max}_{\bm{q}\in\mathcal{P}}\inf_{\bm\theta'\in\bigcup\limits_{m=1}^M \Theta_m \setminus\Theta_{g(\bm\theta)}} \sum\limits_{u=1}^{\abs*{\mathcal{U}}} q_u D_u(\bm\theta||\bm\theta').
\end{equation}
Furthermore, $\bm{q}^*(\bm\theta)$ is continuous at each $\bm\theta$.
\end{proposition}
\begin{proof} See Appendix \ref{appendix:A} for the proof. We assume that the hypothesis sets are such that the maximum is unique. In the case of best arm identification [Theorem 5, \cite{garivier2016optimal}] and anomaly detection [Proposition 3, \cite{gayathri}], the maximum is indeed unique.
\end{proof}

Some remarks are in order: First, the lower bound in \eqref{eq:lb} is non-asymptotic in nature, and so it is a stronger result than
the asymptotic lower bounds generally seen in the literature on controlled sensing and anomaly detection, see, e.g., \cite{chernoff}, \cite{vvv}
and \cite{cohen}. Taking the limit as $\alpha\to 0$, the asymptotic lower bound we get is
\begin{equation}
\mathop{\lim\inf}_{\alpha\to 0} \frac{\mathbb{E}_{\bm\theta}[\tau]}{\abs*{\log\alpha}}\geq \frac{1}{D^*(\bm\theta)},
\end{equation}
which has the same form as the asypmtotic lower bounds generally found in controlled sensing literature. Moreover,
this lower bound is applicable not just to single-parameter exponential distributions, but to general parametrized families.
Intituively, $q_u^*(\bm\theta)$ represents the optimal proportion of the number of times control u should be chosen by a policy that tries
to achieve the lower bound, and $D^*(\bm\theta)$ represents the maximum possible rate of `information' extraction in the worst case
scenario.

\section{Overview of results}\label{sec:overview}
We propose a policy, based on the policy given in \cite{garivier2016optimal}, and show the following properties.

\textbf{Theorem 1.} The proposed policy is an $\bar{\alpha}$-correct policy for any given $\alpha\in(0,1)$.

\textbf{Theorem 2.} [Almost-sure upper bound] For any state of nature $\bm\theta\in\bigcup\limits_{m=1}^{M} \Theta_m$, the proposed policy satisfies
\begin{equation}
\mathbb{P}_{\bm\theta}\bra*{\mathop{\lim\sup}_{\alpha\to 0}\frac{\tau}{\abs*{\log\alpha}}\leq \frac{1}{D^*(\bm\theta)}}=1.
\end{equation}

\textbf{Theorem 3.} [Asymptotic optimality in expectation] For any state of nature $\bm\theta\in\bigcup\limits_{m=1}^{M} \Theta_m$, the proposed policy satisfies
\begin{equation}
\mathop{\lim\sup}_{\alpha\to 0} \frac{\mathbb{E}_{\bm\theta}[\tau]}{\abs*{\log\alpha}}\leq  \frac{1}{D^*(\bm\theta)}.
\end{equation}

Proofs of the above theorems are given in Appendix B. Theorem 3 implies that the proposed policy is asymptotically
optimal. We now discuss two applications of composite multihypothesis controlled sensing. In both applications we assume
that the distributions of observations collected across all controls follow the same single-parameter exponential family.
\begin{enumerate}
\item Best-K arms identification in a multi-armed bandit: The problem of identification of best-K arms in a multi-armed
bandit with minimum expected delay under constraint on error probability, can be cast as a sequential controlled
sensing problem where each control corresponds to an arm and there are $M={\abs*{\mathcal{U}}\choose K}$ hypotheses, such that each
hypothesis consists of a unique combination of $K$ arms which have the highest expectation parameters (we assume
that the statistic $T$ is identity as in \cite{garivier2016optimal}). Since $\dot{A}$ is increasing, we can express any hypothesis set as
\begin{equation}
\Theta=\bigcup\limits_{\pi} \bre*{\bm\theta'\in\Omega : \theta'_{\pi(\phi(1))}\geq\theta'_{\pi(\phi(2))}\geq\dots\geq \theta'_{\pi(\phi(K))},\forall i\notin\bre*{\phi(1),\phi{2},\dots,\phi(K)}}
\end{equation}
where $\bre*{\phi(1),\phi{2},\dots,\phi(K)}$  denotes a unique combination of $K$ controls and $\pi$ denotes a permutation of this
combination. Note that each hypothesis set is open and convex. All hypothesis sets are mutually disjoint. Hence,
Assumptions 1, 2 and 3 hold. It can be verified that Assumption 4 also holds. Thus the proposed policy can be
applied in this scenario and we get an asymptotically optimal policy.

\item  Sequential controlled anomaly detection: The framework of controlled anomaly detection consists of multiple
streams of observations. All distributions are the same except for one stream, which we call as the anomalous
stream. The objective is to sequentially collect observations by choosing one stream in each time step, and find the
anomalous stream in minimum expected delay, while ensuring that the probability of error is bounded by a given
constraint. We can cast the anomaly detection problem as a controlled sensing problem where each control picks
a unique stream to collect observations, and the hypotheses are as follows. Let $M=\abs*{\mathcal{U}}$ and for $m\in\mathcal{M}$,
\begin{equation}
	\Theta_m=\bre*{\bm\theta'\in\Omega : \theta'_i=\theta, \forall i\neq m \text{ for some }\theta \text{ and } \theta'_m\neq \theta}.
\end{equation}
So, hypothesis $m$ indicates that the $m^\text{th}$ stream is anomalous. Observe that each $\Theta_m$ is a 2-D plane (the degrees
of freedom being the anomalous and non-anomalous parameters) without the 1-D line given by
\begin{equation}
L = \bre*{\bm\theta'\in\Omega : \theta'_1=\theta'_2=\dots=\theta'_{\abs*{\mathcal{U}}}}.
\end{equation}

Thus each $\Theta_m$ can be expressed as a union of two convex sets which are open in their own affine hulls, and all
such convex sets that form the hypothesis sets are mutually disjoint. Hence, Assumptions 1, 2 and 3 hold. It can
be verified that Assumption 4 also holds. Thus the proposed policy can be applied in this scenario and we get an
asymptotically optimal policy.

\end{enumerate}

\section{Proposed policy}\label{sec:pp}
Recall that a policy has three essential components: a decision, a control law and a stopping rule. We discuss these components in detail, after introducing the required notation.

Let $N_u(n)$ be the number of times control $u$ is chosen up to time $n$.
\begin{equation}
N_u(n)=\sum\limits_{k=1}^{n} \mathbbm{1}_{\lbrace U_k=u\rbrace}.
\end{equation}
Let $S_u(n)$ be the sum of sufficient statistics of control $u$ up to time $n$.
\begin{equation}
S_u(n)=\sum\limits_{k=1}^{n}T_u(Y_k)\mathbbm{1}_{\bre*{U_k=u}}.
\end{equation}
For all hypothesis $i,j\in\mathcal{M}$, we define the Generalized Likelihood Ratio Test Statistic as
\begin{equation}
Z_{i,j}(n)\coloneqq\log\frac{\sup\limits_{\bm\theta'\in\Theta_i}\prod\limits_{u=1}^{\abs*{\mathcal{U}}} p(\underbar{$Y$}^u(n);\bm\theta',u)}{\sup\limits_{\bm\theta''\in\Theta_j}\prod\limits_{u=1}^{\abs*{\mathcal{U}}} p(\underbar{$Y$}^u(n);\bm\theta'',u)},
\end{equation}
where $\underbar{$Y$}^u(n)=\para*{Y_k : U_k=u, k\leq n}$ is the collection of observations from control $u$.
Let
\begin{equation}
	Z_i(n) \coloneqq \min_{j\neq i,j\in\mathcal{M}} Z_{i,j}(n)\text{ and }Z(n)\coloneqq\max_{i\in\mathcal{M}}Z_i(n).
\end{equation}
Let $\bm\theta^*(n)\in\Omega$ be the global maximum likelihood estimate of $\bm\theta$,
\begin{equation}
\bm\theta^*(n)\coloneqq\mathop{\arg\max}_{\bm\theta'\in\Omega} \prod\limits_{u=1}^{\abs*{\mathcal{U}}} p(\underbar{$Y$}^u(n);\bm\theta',u).
\end{equation}
So, $\forall u\in\mathcal{U}$,
\begin{equation}\label{eq:mle1}
\theta^*_u(n) = b_u\para*{\frac{S_u(n)}{N_u(n)}}.
\end{equation}

\subsection{Stopping time}
We adopt the approach in \cite{garivier2016optimal} and define the stopping time as follows.
\begin{equation}\label{eq:stop}
\tau \coloneqq \inf \lbrace n\in \mathbb{N} : Z(n) \geq \beta(n,\alpha)\rbrace.
\end{equation}
$\beta(n,\alpha)$ is a dynamic threshold given by
\begin{equation}\label{eq:beta}
\beta(n,\alpha)= v(n)+w(\alpha),
\end{equation}
where
\begin{equation}\label{eq:w}
w(\alpha)=\abs*{\log \alpha}+\sqrt{4\abs*{\mathcal{U}}\abs*{\log \alpha}}
\end{equation}
and
\begin{equation}\label{eq:v}
v(n)=C+\log\para*{n(1+\log n)^{\abs*{\mathcal{U}}+2}}+\sqrt{4\abs*{\mathcal{U}}\log\para*{n(1+\log n)^{\abs*{\mathcal{U}}+2}}}.
\end{equation}
Here $C$ is a constant given by $C=2\abs*{\mathcal{U}}\sqrt{2\log\frac{2\abs*{\mathcal{U}}}{e}+\frac{1}{\abs*{\mathcal{U}}}\log\frac{2e^{\abs*{\mathcal{U}}+1}}{\abs*{\mathcal{U}}^{\abs*{\mathcal{U}}}}}+\log\frac{2e^{\abs*{\mathcal{U}}+1}}{\abs*{\mathcal{U}}^{\abs*{\mathcal{U}}}}$.

The threshold $\beta(n,\alpha)$ is based on the deviation inequality given in Theorem 2 in \cite{magureanu2014lipschitz}, which is stated for Bernoulli distributions, but easily extendable to single-parameter exponential family distributions. For sake of completeness, we provide this extension in Appendix \ref{appendix:C}.
\subsection{Decision}
At each time step, the policy's estimate of the true hypothesis will be called as the recommendation at that time step. The recommendation $\hat{r}(n)$ is the nearest hypothesis set $\Theta_m$ to the global MLE $\bm\theta^*(n)$.
\begin{equation}
	\hat{r}(n)\in\mathop{\arg\min}_{m\in\mathcal{M}} \Delta\para*{\bm\theta^*(n),\Theta_m}.
\end{equation}
The decision is given by:
\begin{equation}
\hat{m}\in\mathop{\arg\max}\limits_{m\in\mathcal{M}}Z_m(\tau).
\end{equation}

\subsection{Control Law}
For initialization, all controls are selected once. For the control law, we follow the approach used in the `track-and-stop' strategy, proposed in \cite{garivier2016optimal}. The idea is to choose the control so as to get the empirical proportions $\para*{\frac{N_u(n)}{n}}$ close to the optimal proportions $\bm{q}^*(\bm\theta)$. Since $\bm\theta$ is unknown, we use the plug-in estimates $\bm{q}^*(\hat{\bm\theta}(n))$, where $\hat{\bm\theta}(n)$ is the nearest vector in recommended hypothesis set $\Theta_{\hat{r}(n)}$ to the global MLE $\bm\theta^*(n)$.
\begin{equation}
\hat{\bm\theta}(n)\in\mathop{\arg\min}_{\bm\theta'\in\Theta_{\hat{r}(n)}} \norm*{\bm\theta'-\bm\theta^*(n)}.
\end{equation}
If no minimizer exists, choose $\hat{\bm\theta}(n)$ to be $\rho$-closest of $\bm\theta^*(n)$ in $\Theta_{\hat{r}(n)}$, where $\rho> 1$ is fixed.
\begin{equation}\label{eq:theta_hat}
\norm*{\hat{\bm\theta}(n)-\bm\theta^*(n)}\leq \rho \inf_{\bm\theta'\in\Theta_{\hat{r}(n)}} \norm*{\bm\theta'-\bm\theta^*(n)}=\rho\Delta (\bm\theta^*(n),\Theta_{\hat{r}(n)}).
\end{equation}
Let $\bm{q}^\epsilon(\bm\theta)$ be a $L^\infty$ projection of $q^*(\bm\theta)$ onto $\bre*{\bm{q}\in\mathcal{P}:\forall u\in\mathcal{U}, q_u \in \bra*{\epsilon,1}}$. Then we select the control at time $n+1$ according to
\begin{equation}
	u_{n+1}\in \mathop{\arg\max}_{u\in\mathcal{U}} \sum\limits_{k=1}^n q_u^{\epsilon_k}(\hat{\bm\theta}(k))-N_u(n),
\end{equation}
where $\epsilon_k=\frac{1}{2}(\abs*{\mathcal{U}}^2+k)^{-1/2}$. Note that this projection enforces exploration of the controls in the initial stages when the estimates are not quite accurate. This forced exploration decays as time progresses.

\begin{lemma}[Lemma 7, \cite{garivier2016optimal}]\label{lem:ctrack}
The control law ensures that $\forall n\in\mathbb{N}$ and $\forall u\in\mathcal{U}$,
\begin{equation}
N_u(n)\geq \sqrt{n+\abs*{\mathcal{U}}^2}-2\abs*{\mathcal{U}}
\end{equation}
 and that 
\begin{equation}\label{eq:nbdd}
\max_{u\in\mathcal{U}} \abs*{N_u(n)-\sum\limits_{k=0}^{n-1} q_u^*(\hat{\bm\theta}(k))}\leq \abs*{\mathcal{U}}(1+\sqrt{n}).
\end{equation}
\end{lemma}

Observe that the GLRT statistic has a maximum likelihood in the numerator, which makes it difficult to find a constant threshold such that probability of error can be constrained. In \cite{Vaidhiyan2018}, for example, the authors circumvent this problem by defining a modified GLRT statistic, which has a likelihood averaged over a prior in the numerator instead of the maximum likelihood, and have a constant threshold policy.

\section{Numerical results}\label{sec:nr}
We implemented the proposed policy in a general composite multi-hypothesis detection scenario. The set of controls is $\mathcal{U}=\bre*{1,2,3,4,5}$. The observations from the controls follow normal distributions with means $\bm\mu=\bre*{1,2,12,8,15}$ and variances $\bm\sigma^2=\bre*{1,1,16,4,9}$. The variances are assumed to be known. In this case, the true parameter for control $u$ is $\theta_u=\frac{\mu_u}{\sigma_u}$. So, the true vector of parameters is $\bm\theta=\bre*{1,2,3,4,5}$. The hypothesis are as follows:
\begin{align}
\Theta_1 &= \bre*{ 0\leq\theta_1\leq 2, 1\leq\theta_2\leq 3,2\leq\theta_3\leq 4,3\leq\theta_4\leq 5,4\leq\theta_5\leq 6}\\
\Theta_2 &= \bre*{ 0\leq\theta_1\leq 2,-2\leq\theta_2\leq 0,4\leq\theta_3\leq 6,3\leq\theta_4\leq 5,7\leq\theta_5\leq 9}\\
\Theta_3 &= \bre*{-2\leq\theta_1\leq 0, 1\leq\theta_2\leq 3,2\leq\theta_3\leq 4,5\leq\theta_4\leq 7,2\leq\theta_5\leq 5}\\
\Theta_4 &= \bre*{-2\leq\theta_1\leq 0, 3\leq\theta_2\leq 5,0\leq\theta_3\leq 2,3\leq\theta_4\leq 5,4\leq\theta_5\leq 6}.
\end{align}

Fig. \ref{fig:1} shows the plot of the ratio of empirical mean stopping time to $\abs*{\log (\alpha)}$ versus $\abs*{\log (\alpha)}$, in comparison with the lower bound $\frac{1}{D^*(\bm\theta)}=2.2601$. The empirical mean stopping time is the average of the stopping times obtained in 100 independent iterations.  Observe that the ratio of the empirical mean stopping time to $|\log(\alpha)|$ approaches the lower bound $\frac{1}{D^*(\bm\theta)}$, as $\alpha$ decreases, thereby demonstrating the asymptotic optimality of the proposed policy.
\begin{figure}
\centering
  \includegraphics[width=15 cm]{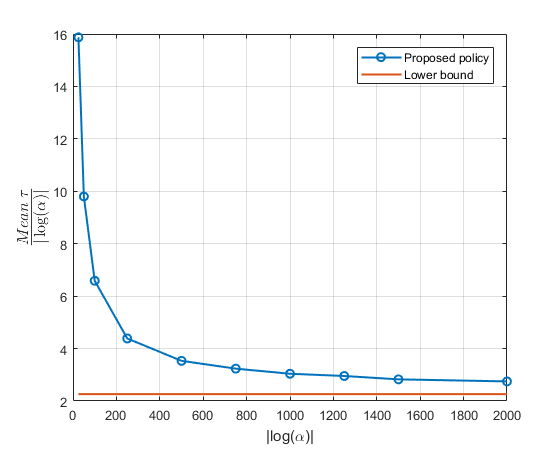}
  \caption{Performance of proposed policy in a general composite hypothesis setting}
  \label{fig:1}
\end{figure}

\begin{appendices}
\section{}\label{appendix:A}
\noindent\textit{Proof of proposition 1.}
Let $\bm\theta\in\bigcup\limits_{m=1}^M \Theta_m$ be fixed, and $\Theta=\bigcup\limits_{m=1}^M \Theta_m \setminus\Theta_{g(\bm\theta)}$. Let $f(\bm{q}):\mathcal{P}\to\mathbb{R}$, such that
\begin{equation}
f(\bm{q})=\inf_{\bm\theta'\in\Theta} \sum\limits_{u=1}^{\abs*{\mathcal{U}}} q_u D_u(\bm\theta||\bm\theta').
\end{equation}
Note that the map $(\bm{q},\bm\theta')\mapsto\sum\limits_{u=1}^U q_u D_u(\bm\theta||\bm\theta')$ is bounded below by 0, and $\Theta$ is non-empty, so $f(\bm{q})$ is well-defined. Let
\begin{equation}\label{eq:ps}
\mathcal{P}_S\coloneqq\bre*{\bm{q}\in\mathcal{P}:\forall u\in S,q_u>0 \text{ and }\forall u\notin S,q_u=0},
\end{equation}
for $S\in 2^\mathcal{U}\setminus\bre*{\phi}$. Let $\mathcal{S}=2^\mathcal{U}\setminus\bre*{\bre*{\phi}\bigcup\limits_{u=1}^{\abs*{\mathcal{U}}}{\bre*{u}}}$. Note that 
\begin{equation}
\mathcal{P}=\bigcup\limits_{S\in 2^\mathcal{U}\setminus\bre*{\phi}}\mathcal{P}_{S}=\bigcup\limits_{S\in\mathcal{S}}\mathcal{P}_S\bigcup\limits_{u=1}^{\abs*{\mathcal{U}}}\mathcal{P}_{\bre*{u}}.
\end{equation}
Note that $f$ is concave on $\mathcal{P}$, since it is an infima of an affine family of functions.
Hence, we have that $f$ is concave on $\mathcal{P}$. Since for any $S\in\mathcal{S}$, $\mathcal{P}_S\subset\mathcal{P}$, $f$ is concave on $\mathcal{P}_S$ and thus continuous on $\text{relint}(\mathcal{P}_S)=\mathcal{P}_S$. We now show that $f$ is lower semi-continuous on $\mathcal{P}$. Consider any $S\in\mathcal{S}$. Let $\bm{q}_0\in\mathcal{P}_S$.  Consider a sequence $\bre*{\bm{q}_n}\subset\mathcal{P}$  such that $\bm{q}_n\to \bm{q}_0$. So we have,
\begin{align}
\mathop{\lim\inf}_{n\to\infty} f(\bm{q}_n)&=\mathop{\lim\inf}_{n\to\infty}\inf_{\bm\theta'\in\Theta} \bre*{\sum\limits_{u\in S} q_{n,u} D_u(\bm\theta||\bm\theta')+\sum\limits_{u\notin S} q_{n,u} D_u(\bm\theta||\bm\theta')}\\
&\geq \mathop{\lim\inf}_{n\to\infty}\bre*{\inf_{\bm\theta'\in\Theta}\sum\limits_{u\in S} q_{n,u} D_u(\bm\theta||\bm\theta')+\inf_{\bm\theta'\in\Theta}\sum\limits_{u\notin S} q_{n,u} D_u(\bm\theta||\bm\theta')}\\
&\geq \mathop{\lim\inf}_{n\to\infty}\bre*{\para*{\sum\limits_{u\in S}q_{n,u}}\inf_{\bm\theta'\in\Theta}\sum\limits_{u\in S} \frac{q_{n,u}}{\para*{\sum\limits_{u\in S} q_{n,u}}} D_u(\bm\theta||\bm\theta')+\sum\limits_{u\notin S}q_{n,u}\inf_{\bm\theta'\in\Theta} D_u(\bm\theta||\bm\theta')}.\label{eq:liminf}
\end{align}
Since $\forall u\notin S, q_{n,u}\to q_{0,u}=0$, we consequently get
\begin{align}
\sum\limits_{u\notin S}q_{n,u}\inf_{\bm\theta'\in\Theta} D_u(\bm\theta||\bm\theta')&\to 0,\label{eq:inf_conti1}\\
\sum\limits_{u\in S}q_{n,u}&\to 1,\\
\forall u\in S,  \frac{q_{n,u}}{\para*{\sum\limits_{u\in S} q_{n,u}}}&\to q_{0,u},\\
\inf_{\bm\theta'\in\Theta}\sum\limits_{u\in S} \frac{q_{n,u}}{\para*{\sum\limits_{u\in S} q_{n,u}}} D_u(\bm\theta||\bm\theta')&\to f(\bm{q}_0).\label{eq:inf_conti2}
\end{align}
Note that \eqref{eq:inf_conti2} follows from the continuity of $f$ on $\mathcal{P}_S$. Applying \eqref{eq:inf_conti1} and \eqref{eq:inf_conti2} in \eqref{eq:liminf}, we get
\begin{equation}
\mathop{\lim\inf}_{n\to\infty} f(\bm{q}_n)\geq f(\bm{q}_0).
\end{equation}
Now consider the singleton element $\bm{q}_0\in \mathcal{P}_{\bre*{v}}$ for any $v\in\mathcal{U}$. Note that $q_{0,u}=1$. Similarly as before, consider a sequence $\bre*{\bm{q}_n}\subset\mathcal{P}$  such that $\bm{q}_n\to \bm{q}_0$. So we have,
\begin{align}
\mathop{\lim\inf}_{n\to\infty} f(\bm{q}_n)&=\mathop{\lim\inf}_{n\to\infty}\inf_{\bm\theta'\in\Theta}\sum\limits_{u=1}^{\abs*{\mathcal{U}}} q_{n,u} D_u(\bm\theta||\bm\theta')\\
&\geq \mathop{\lim\inf}_{n\to\infty}\sum\limits_{u=1}^{\abs*{\mathcal{U}}}q_{n,u}\inf_{\bm\theta'\in\Theta}D_u(\bm\theta||\bm\theta')\\
&=\inf_{\bm\theta'\in\Theta}D_v(\bm\theta||\bm\theta')\\
&=f(\bm{q}_0).
\end{align}
Since $f$ is lower semi-continuous on $\mathcal{P}_S$ for any $S\in\mathcal{S}$ and on $\mathcal{P}_{\bre*{u}}$ for any $u\in\mathcal{U}$, $f$ is lower semi-continuous on $\mathcal{P}$. We now show that $f$ is upper semi-continuous on $\mathcal{P}$. Consider a sequence $\bre*{\bm{q}_n}\subset\mathcal{P}$  such that $\bm{q}_n\to \bm{q}_0\in\mathcal{P}$. From the definition of $f$, it follows that $\exists\bre*{\bm\theta_k}\subset\Theta$ such that 
\begin{equation}
\sum\limits_{u=1}^{\abs*{\mathcal{U}}} q_{0,u} D_u(\bm\theta||\bm\theta_k)\to f(\bm{q}_0).
\end{equation}
So we get,
\begin{align}
\mathop{\lim\sup}_{n\to\infty} f(\bm{q}_n)&=\mathop{\lim\sup}_{n\to\infty}\inf_{\bm\theta'\in\Theta} \sum\limits_{u=1}^{\abs*{\mathcal{U}}} q_{n,u} D_u(\bm\theta||\bm\theta')\\
&\leq\mathop{\lim\sup}_{n\to\infty}  \sum\limits_{u=1}^{\abs*{\mathcal{U}}} q_{n,u} D_u(\bm\theta||\bm\theta_k)\\
&=\sum\limits_{u=1}^{\abs*{\mathcal{U}}} q_{0,u} D_u(\bm\theta||\bm\theta_k).
\end{align}
Note that this holds for all $k\in\mathbb{N}$. Thus, taking limit $k\to\infty$, we get
\begin{equation}
\mathop{\lim\sup}_{n\to\infty} f(\bm{q}_n)\leq f(\bm{q}_0).
\end{equation}
Hence, $f$ is upper semi-continuous on $\mathcal{P}$. Since $f$ is both upper and lower semi-continuous on $\mathcal{P}$, we conclude $f$ is continuous on $\mathcal{P}$. Since $\mathcal{P}$ is compact, $f$ achieves the maximum value $D^*(\bm\theta)$ on $\mathcal{P}$.

We now show that the function $f_0(\bm{q},\bm{\theta}):\mathcal{P}\times\bigcup\limits_{m=1}^M \Theta_m\to\mathbb{R}$ given by
\begin{equation}
f_0(\bm{q},\bm{\theta})=\inf_{\bm\theta'\in\bigcup\limits_{m=1}^M \Theta_m \setminus\Theta_{g(\bm\theta)}} \sum\limits_{u=1}^{\abs*{\mathcal{U}}} q_u D_u(\bm\theta||\bm\theta'),
\end{equation}
is continuous. Let $m'\in\mathcal{M}$, $i\in\bre*{1,2,...,x_{m'}}$, and $\bm{q}\in\mathcal{P}$ be fixed. First, we show that the functions $f^{(m)}_j:\Gamma^{(i)}_{m'}\to\mathbb{R}$ given by
\begin{equation}
f^{(m)}_j(\bm\theta)=\inf\limits_{\bm\theta'\in\Gamma^{(j)}_m} \sum\limits_{u=1}^{\abs*{\mathcal{U}}} q_u D_u(\bm\theta||\bm\theta')
\end{equation}
are continuous, where $m\in\mathcal{M}$ and $j\in\bre*{1,2,\dots,x_m}$. We show that $f^{(m)}_j$ is convex as follows. Let $\bm\theta_1,\bm\theta_2\in\Gamma^{(i)}_{m'}$ and $\lambda\in [0,1]$. So for any $\bm\theta'_1,\bm\theta'_2\in\Gamma^{(j)}_m$,
\begin{align}
f^{(m)}_j(\lambda\bm\theta_1+(1-\lambda)\bm\theta_2)&\leq \sum\limits_{u=1}^{\abs*{\mathcal{U}}} q_u D_u(\lambda\bm\theta_1+(1-\lambda)\bm\theta_2||\lambda\bm\theta'_1+(1-\lambda)\bm\theta'_2)\\
&\leq \sum\limits_{u=1}^{\abs*{\mathcal{U}}} q_u \bra*{\lambda D_u(\bm\theta_1||\bm\theta'_1)+(1-\lambda)D_u(\bm\theta_2||\bm\theta'_2)}\\
&=\lambda\sum\limits_{u=1}^{\abs*{\mathcal{U}}} q_u D_u(\bm\theta_1||\bm\theta'_1) + (1-\lambda)\sum\limits_{u=1}^{\abs*{\mathcal{U}}} q_u D_u(\bm\theta_2||\bm\theta'_2)
\end{align}
This holds due to the convexity of $D_u$. Taking infimum over $\bm\theta'_1$ and $\bm\theta'_2$, we get
\begin{equation}
f^{(m)}_j(\lambda\bm\theta_1+(1-\lambda)\bm\theta_2)\leq \lambda f^{(m)}_j(\bm\theta_1)+(1-\lambda) f^{(m)}_j(\bm\theta_2).
\end{equation}
Thus, $f^{(m)}_j$ is convex on $\Gamma^{(i)}_{m'}$ and hence continuous on $\Gamma^{(i)}_{m'}$, since $\Gamma^{(i)}_{m'}=\text{relint}\para*{\Gamma^{(i)}_{m'}}$. This further implies that $\min\limits_{m\neq m'}\min\limits_{j\in [x_m]}f^{(m)}_j$ is continuous on $\Gamma^{(i)}_{m'}$. Since $m'\in\mathcal{M}$, $i\in [x_{m'}]$ were chosen arbitrarily and from assumption \ref{assump:4} on the structure of $\Theta_m$'s, we get that for a fixed $\bm{q}\in\mathcal{P}$, the function $f^*:\bigcup\limits_{m=1}^M \Theta_m\to\mathbb{R}$ given by
\begin{equation}
f^*(\bm\theta)=\min_{m\neq g(\bm\theta)}\min\limits_{j\in [x_m]}\inf_{\bm\theta'\in\Gamma^{(j)}_m} \sum\limits_{u=1}^{\abs*{\mathcal{U}}}q_u D_u(\bm\theta||\bm\theta')=\inf_{\bm\theta'\in\bigcup\limits_{m=1}^M \Theta_m \setminus\Theta_{g(\bm\theta)}} \sum\limits_{u=1}^{\abs*{\mathcal{U}}} q_u D_u(\bm\theta||\bm\theta')
\end{equation}
is continuous on its domain. We prove the continuity of $f_0(\bm{q},\bm\theta)$ by using the continuity of $f(\bm{q})$ and $f^*(\bm\theta)$. Consider a sequence $\bre*{\bm{q}_n}\subset\mathcal{P}$  such that $\bm{q}_n\to \bm{q}_0\in\mathcal{P}$ and a sequence $\bre*{\bm\theta_n}\subset\bigcup\limits_{m=1}^M \Theta_m$ such that $\bm\theta_n\to\bm\theta_0\in\bigcup\limits_{m=1}^M \Theta_m$. Let $\bm{q}_0\in\mathcal{P}_S$. Note that $\exists N_1$ such that $\forall n\geq N_1$, $\bm\theta_n\in\Theta_{g(\bm\theta_0)}$.  Given $\epsilon\in\para*{0,\min\limits_{u\in S}\frac{q_{0,u}}{2}}$, $\exists N_2$ such that $\forall u\in\mathcal{U}, \forall n\geq N_2$, $q_{n,u}\geq q_{0,u}-\epsilon$. Let $\Theta=\bigcup\limits_{m=1}^M \Theta_m \setminus\Theta_{g(\bm\theta_0)}$. Thus $\forall n\geq \max(N_1,N_2)$,
\begin{equation}
f_0(\bm{q}_n,\bm\theta_n)\geq (1-\abs*{S}\epsilon)\inf\limits_{\bm\theta'\in\Theta} \sum\limits_{u\in S}\frac{q_{0,u}-\epsilon}{1-\abs*{S}\epsilon}D_u(\bm\theta_n||\bm\theta').
\end{equation}
By continuity of $f^*$, we get
\begin{align}
\mathop{\lim\inf}_{n\to\infty} f_0(\bm{q}_n,\bm\theta_n)&\geq \mathop{\lim\inf}_{n\to\infty}(1-\abs*{S}\epsilon)\inf\limits_{\bm\theta'\in\Theta} \sum\limits_{u\in S}\frac{q_{0,u}-\epsilon}{1-\abs*{S}\epsilon}D_u(\bm\theta_n||\bm\theta')\\
&=(1-\abs*{S}\epsilon)\inf\limits_{\bm\theta'\in\Theta} \sum\limits_{u\in S}\frac{q_{0,u}-\epsilon}{1-\abs*{S}\epsilon}D_u(\bm\theta_0||\bm\theta').
\end{align}
By continuity of $f$ and letting $\epsilon\to 0$, we get
\begin{equation}
\mathop{\lim\inf}_{n\to\infty} f_0(\bm{q}_n,\bm\theta_n)\geq \inf\limits_{\bm\theta'\in\Theta} \sum\limits_{u\in S} q_{0,u} D_u(\bm\theta_0||\bm\theta')=f_0(\bm{q}_0,\bm\theta_0).
\end{equation}
Let $\bre*{\bm\theta'_k}\subset\Theta$ such that 
\begin{equation}
\sum\limits_{u=1}^{\abs*{\mathcal{U}}} q_{0,u} D_u(\bm\theta_0||\bm\theta'_k)\to\inf_{\bm\theta'\in\Theta}\sum\limits_{u=1}^{\abs*{\mathcal{U}}} q_{0,u} D_u(\bm\theta_0||\bm\theta')=f_0(\bm{q}_0,\bm\theta_0).
\end{equation}
Thus,
\begin{align}
\mathop{\lim\sup}_{n\to\infty} f_0(\bm{q}_n,\bm\theta_n)&\leq \mathop{\lim\sup}_{n\to\infty}\sum\limits_{u=1}^{\abs*{\mathcal{U}}} q_{n,u} D_u(\bm\theta_n||\bm\theta'_k)\\
&\leq  \sum\limits_{u=1}^{\abs*{\mathcal{U}}} q_{0,u} D_u(\bm\theta_0||\bm\theta'_k).
\end{align}
Taking limit as $k\to\infty$, we get
\begin{equation}
\mathop{\lim\sup}_{n\to\infty} f_0(\bm{q}_n,\bm\theta_n)\leq f_0(\bm{q}_0,\bm\theta_0).
\end{equation}
Hence, we conclude that $f_0$ is continuous everywhere on its domain. Continuity of $\bm{q}^*$ follows from Berge's maximum theorem.
\begin{proposition}\label{prop:lambert}
Let $\beta_0(n,\alpha)$ satisfy the following equation.
\begin{equation}\label{eq:lambert1}
e^{\beta_0(n,\alpha)}=\frac{4e^{\abs*{\mathcal{U}}+1}}{\alpha \abs*{\mathcal{U}}^{\abs*{\mathcal{U}}}}\beta_0(n,\alpha)^{2\abs*{\mathcal{U}}}n(1+\log (n))^{\abs*{\mathcal{U}}+2}.
\end{equation}
An upper bound on $\beta_0(n,\alpha)$ is $\beta(n,\alpha)$ as given in \eqref{eq:beta}. Consequently,
\begin{equation}\label{eq:lambert2}
e^{\beta(n,\alpha)}\geq\frac{4e^{\abs*{\mathcal{U}}+1}}{\alpha \abs*{\mathcal{U}}^{\abs*{\mathcal{U}}}}\beta(n,\alpha)^{2\abs*{\mathcal{U}}}n(1+\log (n))^{\abs*{\mathcal{U}}+2}.
\end{equation}

\end{proposition}
\begin{proof}
This result follows from Theorem 1 in \cite{lambert} and expressing \eqref{eq:lambert1} in terms of the Lambert W-function.
\end{proof}
\section{}\label{appendix:B}
In this Appendix, we prove the theorems stated in section \ref{sec:overview}. We first establish some asymptotic convergence results.
\begin{proposition}\label{prop:conv}
Let $\bm\theta\in\bigcup\limits_{m=1}^M \Theta_m$ be the state of nature vector of parameters. Then the following holds for the policy that never stops and uses the proposed policy's recommendation and control law
\begin{align}
	\bm\theta^*(n)&\mathop{\to}^{\text{a.s.}}\bm\theta,\label{eq:p1}\\
	\hat{r}(n)&\mathop{\to}^{\text{a.s.}} g(\bm\theta),\label{eq:p2}\\
	\hat{\bm\theta}(n)&\mathop{\to}^{\text{a.s.}}\bm\theta,\label{eq:p3}\\
	\bm{q}^*(\hat{\bm\theta}(n))&\mathop{\to}^{\text{a.s.}}\bm{q}^*(\bm\theta),\label{eq:p4}\\
	\frac{N_u(n)}{n}&\mathop{\to}^{\text{a.s.}}q^*_u(\bm\theta), \forall u\in\mathcal{U},\label{eq:p5}\\
\end{align}
\end{proposition}
\begin{proof}
We have from Lemma \ref{lem:ctrack} that $N_u(n)\mathop{\to}\limits^{\text{a.s.}}\infty$. By the Strong Law of Large Numbers and continuity of $\dot{b}_u$, we get that $\forall u\in\mathcal{U}$
\begin{equation}
\frac{S_u(n)}{N_u(n)}\to\dot{A}_u(\theta_u)\text{ and }\dot{b}_u\para*{\frac{S_u(n)}{N_u(n)}}\mathop{\to}^{\text{a.s.}}\theta_u.
\end{equation}
Thus, \eqref{eq:p1} holds. Note that $\Delta(\bm{x},\Theta_m)$ is continuous at every $\bm{x}\in\Omega$ for any $m\in\mathcal{M}$. Consequently, \eqref{eq:p1} implies that 
\begin{align}\label{eq:theta_star_distance}
\Delta\para*{\bm\theta^*(n),\Theta_{g(\bm\theta)}}&\mathop{\to}^{\text{a.s.}}\Delta(\bm\theta,\Theta_{g(\bm\theta)})=0,\text{ and}\\ 
\Delta\para*{\bm\theta^*(n),\Theta_i}&\mathop{\to}^{\text{a.s.}}\Delta(\bm\theta,\Theta_i)>0, \forall i\neq g(\bm\theta).
\end{align}
Thus, \eqref{eq:p2} holds. Consequently, \eqref{eq:p3} follows from \eqref{eq:p1}, \eqref{eq:p2}, \eqref{eq:theta_star_distance} and \eqref{eq:theta_hat}. Note that \eqref{eq:p4} holds due to \eqref{eq:p3} and continuity of $\bm{q}^*$ (proposition \ref{prop:q}).  Note that it follows from lemma \ref{lem:ctrack} that $\forall u\in \mathcal{U}$,
\begin{equation}\label{eq:i1}
\abs*{\frac{N_u(n)}{n}-\frac{1}{n}\sum\limits_{k=0}^{n-1} q_u^*(\hat{\bm\theta}(k))}\leq \frac{\abs*{\mathcal{U}}(1+\sqrt{n})}{n}\mathop{\to}^{\text{a.s.}}0.
\end{equation}
Using Cesaro's lemma and \eqref{eq:p4}, we have that $\forall u\in\mathcal{U}$,
\begin{equation}\label{eq:i2}
\frac{1}{n}\sum\limits_{k=0}^{n-1} q_u^*(\hat{\bm\theta}(k))\mathop{\to}^{\text{a.s.}}q^*_u(\bm\theta).
\end{equation}
Thus, \eqref{eq:p5} follows from \eqref{eq:i1} and \eqref{eq:i2}.
\end{proof}

\begin{lemma}\label{lem:Zconv} Let $\bm\theta\in\bigcup\limits_{m=1}^M \Theta_m$ be the state of nature vector of parameters. Then the following holds for the policy that never stops and uses the proposed policy's recommendation and control law
\begin{equation}\label{eq:p6}
\frac{Z_{g(\bm\theta)}(n)}{n}\mathop{\to}^{\text{a.s.}}D^*(\bm\theta).
\end{equation}
\end{lemma}
\begin{proof}
Let $\mathcal{E}$ be the event given by
\begin{equation}
\mathcal{E}=\bre*{\forall u\in \mathcal{U},\frac{S_u(n)}{N_u(n)}\to\dot{A}_u(\theta_u)\bigcap\frac{N_u(n)}{n}\to q^*_u(\bm\theta)}.
\end{equation}
By the Strong Law of Large Numbers and \eqref{eq:p5}, we have that $\mathbb{P}_{\bm\theta}[\mathcal{E}]=1$.
\begin{claim}\label{claim:1} $\forall i\in\mathcal{M}$,
\begin{equation}
\frac{1}{n}\log\frac{\prod\limits_{u=1}^{\abs*{\mathcal{U}}}p(\underbar{$Y$}^u(n);\bm\theta,u)}{\sup\limits_{\bm\theta'\in\Theta_i}\prod\limits_{u=1}^{\abs*{\mathcal{U}}} p(\underbar{$Y$}^u(n);\bm\theta',u)}\mathop{\to}^{\text{a.s.}}\inf_{\bm\theta'\in\Theta_i}\sum\limits_{u=1}^{\abs*{\mathcal{U}}} q^*_u(\bm\theta) D_u(\bm\theta||\bm\theta').
\end{equation}
\end{claim}
\begin{proof}
Let $i\in\mathcal{M}$. Since $\log$ is continuous and increasing, we have
\begin{align}
\frac{1}{n}\log\frac{\prod\limits_{u=1}^{\abs*{\mathcal{U}}}p(\underbar{$Y$}^u(n);\bm\theta,u)}{\sup\limits_{\bm\theta'\in\Theta_i}\prod\limits_{u=1}^{\abs*{\mathcal{U}}} p(\underbar{$Y$}^u(n);\bm\theta',u)}&=\inf\limits_{\bm\theta'\in\Theta_i}\frac{1}{n}\sum\limits_{u=1}^{\abs*{\mathcal{U}}}\log\frac{p(\underbar{$Y$}^u(n);\bm\theta,u)}{p(\underbar{$Y$}^u(n);\bm\theta',u)}\\
&= \inf\limits_{\bm\theta'\in\Theta_i} \frac{1}{n}\sum\limits_{u=1}^{\abs*{\mathcal{U}}}\theta_u S_u(n)-N_u(n)A_u(\theta_u)-\bra*{\theta'_u S_u(n)-N_u(n)A_u(\theta'_u)}\\
&= \inf\limits_{\bm\theta'\in\Theta_i} \sum\limits_{u=1}^{\abs*{\mathcal{U}}}\frac{N_u(n)}{n}\bra*{D_u(\bm\theta||\bm\theta')+(\theta_u-\theta'_u)\para*{\frac{S_u(n)}{N_u(n)}-\dot{A}_u(\theta_u)}}\\
&=\inf\limits_{\bm\theta'\in\Theta_i} \sum\limits_{u=1}^{\abs*{\mathcal{U}}}\frac{N_u(n)}{n}D_u(\bm\theta||\bm\theta')+\bm{W}_n^T(\bm\theta-\bm\theta'),
\end{align}
where $\bm{W}_n$ is the $\abs*{\mathcal{U}}$-dimensional vector such that $W_{n,u}=\frac{N_u(n)}{n} \para*{\frac{S_u(n)}{N_u(n)}-\dot{A}_u(\theta_u)}$.
Note that $\exists\bre*{\bm\theta_k}\subset\Theta_i$ such that 
\begin{equation}
\sum\limits_{u=1}^{\abs*{\mathcal{U}}} q^*_u(\bm\theta) D_u(\bm\theta||\bm\theta_k)\to\inf_{\bm\theta'\in\Theta_i}\sum\limits_{u=1}^{\abs*{\mathcal{U}}} q^*_u(\bm\theta) D_u(\bm\theta||\bm\theta').
\end{equation}
Thus on $\mathcal{E}$, we get
\begin{align}
\mathop{\lim\sup}_{n\to\infty} \frac{1}{n}\log\frac{\prod\limits_{u=1}^{\abs*{\mathcal{U}}}p(\underbar{$Y$}^u(n);\bm\theta,u)}{\sup\limits_{\bm\theta'\in\Theta_i}\prod\limits_{u=1}^{\abs*{\mathcal{U}}} p(\underbar{$Y$}^u(n);\bm\theta',u)}&\leq  \mathop{\lim\sup}_{n\to\infty} \sum\limits_{u=1}^{\abs*{\mathcal{U}}}\frac{N_u(n)}{n}D_u(\bm\theta||\bm\theta_k)+\bm{W}_n^T(\bm\theta-\bm\theta_k)\\
&= \sum\limits_{u=1}^{\abs*{\mathcal{U}}} q^*_u(\bm\theta) D_u(\bm\theta||\bm\theta_k).
\end{align}
Note that this holds for all $k\in\mathbb{N}$. Thus, taking limit $k\to\infty$, we get
\begin{equation}\label{eq:limsupz}
\mathop{\lim\sup}_{n\to\infty} \frac{1}{n}\log\frac{\prod\limits_{u=1}^{\abs*{\mathcal{U}}}p(\underbar{$Y$}^u(n);\bm\theta,u)}{\sup\limits_{\bm\theta'\in\Theta_i}\prod\limits_{u=1}^{\abs*{\mathcal{U}}} p(\underbar{$Y$}^u(n);\bm\theta',u)}\leq \inf_{\bm\theta'\in\Theta_i}\sum\limits_{u=1}^{\abs*{\mathcal{U}}} q^*_u(\bm\theta) D_u(\bm\theta||\bm\theta').
\end{equation}
Let $\bm{q}^*(\bm\theta)\in\mathcal{P}_S$ for some $S\in2^\mathcal{U}\setminus\bre*{\phi}$, where $\mathcal{P}_S$ is as defined in \eqref{eq:ps}, We then have,
\begin{equation} \label{eq:rhs}
 \inf\limits_{\bm\theta'\in\Theta_i} \sum\limits_{u=1}^{\abs*{\mathcal{U}}}\frac{N_u(n)}{n}D_u(\bm\theta||\bm\theta')+\bm{W}_n^T(\bm\theta-\bm\theta')\geq \inf\limits_{\bm\theta'\in\Theta_i}  \sum\limits_{u\in S}\frac{N_u(n)}{n}D_u(\bm\theta||\bm\theta')+\bm{W}_n^T(\bm\theta-\bm\theta').
\end{equation}
Let $\epsilon\in\para*{0,\min\limits_{u\in S}q^*_u(\bm\theta)/2}$. Thus on $\mathcal{E}$, $\exists N\in\mathbb{N}$ such that $\forall u\in S$, $\forall n>N$, $\frac{N_u(n)}{n}\geq q^*_u(\bm\theta)-\epsilon>0$. So $\forall n\geq N$, R.H.S of \eqref{eq:rhs} is bounded below
\begin{equation}
\inf\limits_{\bm\theta'\in\Theta_i}\sum\limits_{u\in S}\para*{q^*_u(\bm\theta)-\epsilon}D_u(\bm\theta||\bm\theta') +\bm{W}_n^T(\bm\theta-\bm\theta').
\end{equation}
Note that the function $l:\mathbb{R}^{\abs*{\mathbb{U}}}\to\mathbb{R}$ given by
\begin{equation}
l(\bm{w})= \inf\limits_{\bm\theta'\in\Theta_i} \sum\limits_{u\in S}\para*{q^*_u(\bm\theta)-\epsilon}D_u(\bm\theta||\bm\theta') +\bm{w}^T(\bm\theta-\bm\theta')
\end{equation}
is well-defined and concave on $\mathbb{R}^{\abs*{\mathbb{U}}}$ and thus continuous at $\bm{0}$. Using this we get that, on $\mathcal{E}$,
\begin{align}
\mathop{\lim\inf}_{n\to\infty} \frac{1}{n}\log\frac{\prod\limits_{u=1}^{\abs*{\mathcal{U}}}p(\underbar{$Y$}^u(n);\bm\theta,u)}{\sup\limits_{\bm\theta'\in\Theta_i}\prod\limits_{u=1}^{\abs*{\mathcal{U}}} p(\underbar{$Y$}^u(n);\bm\theta',u)}&\geq \mathop{\lim\inf}_{n\to\infty} l(\bm{W}_n)\\
&=l(\bm{0})\\
&=\inf\limits_{\bm\theta'\in\Theta_i} \sum\limits_{u\in S}\para*{q^*_u(\bm\theta)-\epsilon}D_u(\bm\theta||\bm\theta')\\
&=(1-\abs*{S}\epsilon)\inf\limits_{\bm\theta'\in\Theta_i} \sum\limits_{u\in S}\frac{q^*_u(\bm\theta)-\epsilon}{1-\abs*{S}\epsilon}D_u(\bm\theta||\bm\theta').
\end{align}
Note that this holds for any $\epsilon\in\para*{0,\min\limits_{u\in S}q^*_u(\bm\theta)/2}$. Taking limit as $\epsilon\to0$ and by continuity of $f$ from proposition \ref{prop:q}, we get
\begin{equation}\label{eq:liminfz}
\mathop{\lim\inf}_{n\to\infty} \frac{1}{n}\log\frac{\prod\limits_{u=1}^{\abs*{\mathcal{U}}}p(\underbar{$Y$}^u(n);\bm\theta,u)}{\sup\limits_{\bm\theta'\in\Theta_i}\prod\limits_{u=1}^{\abs*{\mathcal{U}}} p(\underbar{$Y$}^u(n);\bm\theta',u)}\geq \inf_{\bm\theta'\in\Theta_i}\sum\limits_{u=1}^{\abs*{\mathcal{U}}} q^*_u(\bm\theta) D_u(\bm\theta||\bm\theta').
\end{equation}
From \eqref{eq:limsupz} and \eqref{eq:liminfz}, we get the desired claim.
\end{proof}
Now using claim \ref{claim:1}, we get that
\begin{align}
\frac{Z_{g(\bm\theta)}(n)}{n} &= \min\limits_{i\neq g(\bm\theta)} \frac{Z_{g(\bm\theta),i}(n)}{n}\\
&= \min\limits_{i\neq g(\bm\theta)}\frac{1}{n}\log\frac{\sup\limits_{\bm\theta''\in\Theta_{g(\bm\theta)}}\prod\limits_{u=1}^{\abs*{\mathcal{U}}} p(\underbar{$Y$}^u(n);\bm\theta'',u)}{\sup\limits_{\bm\theta'\in\Theta_i}\prod\limits_{u=1}^{\abs*{\mathcal{U}}} p(\underbar{$Y$}^u(n);\bm\theta',u)}\\
&=\min\limits_{i\neq g(\bm\theta)}\bra*{\frac{1}{n}\log\frac{\prod\limits_{u=1}^{\abs*{\mathcal{U}}}p(\underbar{$Y$}^u(n);\bm\theta,u)}{\sup\limits_{\bm\theta'\in\Theta_i}\prod\limits_{u=1}^{\abs*{\mathcal{U}}} p(\underbar{$Y$}^u(n);\bm\theta',u)}-\frac{1}{n}\log\frac{\prod\limits_{u=1}^{\abs*{\mathcal{U}}}p(\underbar{$Y$}^u(n);\bm\theta,u)}{\sup\limits_{\bm\theta''\in\Theta_{g(\bm\theta)}}\prod\limits_{u=1}^{\abs*{\mathcal{U}}} p(\underbar{$Y$}^u(n);\bm\theta'',u)}}\\
&\mathop{\to}^{\text{a.s.}}\min\limits_{i\neq g(\bm\theta)}\bra*{\inf_{\bm\theta'\in\Theta_i}\sum\limits_{u=1}^{\abs*{\mathcal{U}}} q^*_u(\bm\theta) D_u(\bm\theta||\bm\theta')-\inf_{\bm\theta''\in\Theta_{g(\bm\theta)}}\sum\limits_{u=1}^{\abs*{\mathcal{U}}} q^*_u(\bm\theta) D_u(\bm\theta||\bm\theta'')}\\
&=\min\limits_{i\neq g(\bm\theta)}\inf_{\bm\theta'\in\Theta_i}\sum\limits_{u=1}^{\abs*{\mathcal{U}}} q^*_u(\bm\theta) D_u(\bm\theta||\bm\theta')\\
&=\inf_{\bm\theta'\in\bigcup\limits_{m=1}^M \Theta_m \setminus\Theta_{g(\bm\theta)}} \sum\limits_{u=1}^{\abs*{\mathcal{U}}} q^*_u(\bm\theta) D_u(\bm\theta||\bm\theta')\\
&=D^*(\bm\theta).
\end{align}
\end{proof}

We proceed to show that the proposed policy is a $\bar{\alpha}$-correct policy. Observe that from Proposition \ref{prop:conv}, we get that $Z(n)$ is at-least linear in $n$ almost surely for large $n$. On the other hand, the threshold $\beta(n,\alpha)$ is $O(\log n)$. Therefore, the proposed policy stops in finite time almost surely. To prove that the error probability is bounded by $\alpha$, we use a concentration type inequality tailored for single parameter exponential families (Refer to Appendix \ref{appendix:B} for details). We now rigourously prove these claims in the next theorem.
\begin{theorem}
The proposed policy is a $\bar{\alpha}$-correct policy. 
\end{theorem}
\begin{proof}
We first prove that the proposed policy described has a finite stopping rule almost surely. Let $g(\bm\theta)=i$. Consider the event $\mathcal{E}=\bre*{\frac{Z_i(n)}{n}\to D^*(\bm\theta)}$. From Lemma \ref{lem:Zconv}, we have that this event is of probability 1, that is $\mathbb{P}_{\bm\theta}[\mathcal{E}]=1$. Let $\alpha\in(0,1)$.  Let $\epsilon>0$. On $\mathcal{E}$, $\exists N\in \mathbb{N}$ such that $\forall n>N$, 
\begin{equation}
Z(n)\geq Z_i(n)\geq \frac{nD^*(\bm\theta)}{(1+\epsilon)}.
\end{equation}
Consequently,
\begin{align}
\tau&=\inf\bre*{n\in \mathbb{N}:Z(n)\geq \beta(n,\alpha)}\\
&\leq N\lor \inf\bre*{n\in \mathbb{N}:\frac{nD^*(\bm\theta)}{(1+\epsilon)}\geq \beta(n,\alpha)}\\
&\leq N\lor \inf\bre*{n\in \mathbb{N}:\frac{nD^*(\bm\theta)}{(1+\epsilon)}\geq v(n) +w(\alpha)}
\end{align}
where $v$ and $w$ are as given in proposition \ref{prop:lambert}. Note that $\lim\limits_{t\to\infty}v'(t)=0$. Hence,
\begin{equation}
\inf\bre*{n\in \mathbb{N}:\frac{nD^*(\bm\theta)}{(1+\epsilon)}\geq v(n) +w(\alpha)}<\infty.
\end{equation}
Consequently, $\tau<\infty$. Since $\mathbb{P}_{\bm\theta}[\mathcal{E}]=1$, we get $\mathbb{P}_{\bm\theta}[\tau<\infty]=1$.
We first establish an upper bound on $Z_{j,i}$ for any $j\in\mathcal{M}$
\begin{align}
Z_{j,i}(n)&=\log\frac{\sup\limits_{\bm\theta'\in\Theta_j}\prod\limits_{u=1}^{\abs*{\mathcal{U}}} p(\underbar{$Y$}^u(n);\bm\theta',u)}{\sup\limits_{\bm\theta''\in\Theta_i}\prod\limits_{u=1}^{\abs*{\mathcal{U}}} p(\underbar{$Y$}^u(n);\bm\theta'',u)}\\
&\leq \log\frac{\sup\limits_{\bm\theta'\in\Omega}\prod\limits_{u=1}^{\abs*{\mathcal{U}}} p(\underbar{$Y$}^u(n);\bm\theta',u)}{\prod\limits_{u=1}^{\abs*{\mathcal{U}}} p(\underbar{$Y$}^u(n);\bm\theta,u)}\\
&= \sum\limits_{u=1}^{\abs*{\mathcal{U}}} N_u(n)D(\theta^*_u(n)||\theta_u).
\end{align}
We now proceed to prove that error probability is bounded by chosen $\alpha$.
\begin{align}
\mathbb{P}_{\bm\theta} [\hat{m}\neq i]&\leq \mathbb{P}_{\bm\theta}\bigg[\exists n\in\mathbb{N},\min\limits_{j\neq i}Z_{j,i}(n)\geq \beta(n,\alpha)\bigg]\\
&\leq \mathbb{P}_{\bm\theta} \bigg[\exists n\in \mathbb{N},\exists j\in \mathcal{M}\setminus i : Z_{j,i}(n)\geq \beta(n,\alpha)\bigg]\\
&\leq \sum\limits_{n=1}^{\infty} \mathbb{P}_{\bm\theta}\bigg[  \sum\limits_{u=1}^{\abs*{\mathcal{U}}} N_u(n)D(\theta^*_u(n)||\theta_u)\geq  \beta(n,\alpha)\bigg]\\
&\leq \sum\limits_{n=1}^\infty  2e^{-\beta}\para*{\frac{\beta\ceil*{\beta\log n}}{\abs*{\mathcal{U}}}}^{\abs*{\mathcal{U}}} e^{\abs*{\mathcal{U}}+1}\label{eq:prob_bdd_1}\\
&\leq \alpha \sum\limits_{n=1}^\infty \frac{1}{2n(1+\log n)^2}\label{eq:prob_bdd_2}\\
&\leq \alpha.
\end{align}
The inequality \eqref{eq:prob_bdd_1} follows from Theorem \ref{thm:conc} and \eqref{eq:prob_bdd_2} follows from Proposition \ref{prop:lambert}.
\end{proof}

\begin{theorem}[Almost-sure upper bound]
Let $\bm\theta\in\bigcup\limits_{m=1}^M \Theta_m$. The proposed policy satisfies
\begin{equation}\label{eq:asymp1}
\mathbb{P}_{\bm\theta}\left[ \mathop{\lim\sup}_{\alpha\to 0}\frac{\tau}{\abs*{\log \alpha}}\leq \frac{1}{D^*(\bm\theta)} \right]=1.
\end{equation}
\end{theorem}
\begin{proof}
Let $g(\bm\theta)=i$. Consider the event $\mathcal{E}=\bre*{\frac{Z_i(n)}{n}\to D^*(\bm\theta)}$. From proposition \ref{prop:conv}, we have that this event is of probability 1, that is $\mathbb{P}_{\bm\theta}[\mathcal{E}]=1$. Let $\alpha\in(0,1)$.  Let $\epsilon>0$. On $\mathcal{E}$, $\exists N\in \mathbb{N}$ such that $\forall n>N$, 
\begin{equation}
Z(n)\geq Z_i(n)\geq \frac{nD^*(\bm\theta)}{(1+\epsilon)}.
\end{equation}
Consequently,
\begin{align}
\tau&=\inf\bre*{n\in \mathbb{N}:Z(n)\geq \beta(n,\alpha)}\\
&\leq N\lor \inf\bre*{n\in \mathbb{N}:\frac{nD^*(\bm\theta)}{(1+\epsilon)}\geq \beta(n,\alpha)}\\
&\leq N\lor \inf\bre*{n\in \mathbb{N}:\frac{nD^*(\bm\theta)}{(1+\epsilon)}\geq v(n) +w(\alpha)}
\end{align}
where $v$ and $w$ are as defined in \eqref{eq:v} and \eqref{eq:w} respectively. Note that $\forall t>1, v'(t)>0$ and $v''(t)<0$. Also, $\lim\limits_{t\to\infty}v'(t)=0$. Thus $\exists  N_1\in \mathbb{N}$ such that $\forall n\geq N_1$, $\frac{nD^*(\bm\theta)}{(1+\epsilon)}> v(n)$. Also, $\exists N_2\in\mathbb{N}$ such that $\forall n\geq N_2$, $v'(n)\in\frac*{-\frac{0.5D^*(\bm\theta)}{(1+\epsilon)},\frac{0.5D^*(\bm\theta)}{(1+\epsilon)}}$. Let $N_3=\max\bre*{N,N_1,N_2}$. Note that $N_3$ is not dependent on $\alpha$. So, we get $\forall n\geq N_3$,
\begin{equation}
\tau\leq n+\frac{w(\alpha)}{\frac{D^*(\bm\theta)}{(1+\epsilon)}-v'(n)}.
\end{equation}
Consequently, $\forall n\geq N_3$,
\begin{equation}
\mathop{\lim\sup}_{\alpha\to 0}\frac{\tau}{\abs*{\log \alpha}}\leq \mathop{\lim\sup}_{\alpha\to 0}\frac{w(\alpha)}{\bra*{\frac{D^*(\bm\theta)}{(1+\epsilon)}-v'(n)}\abs*{\log \alpha}}.
\end{equation}
Note that $\lim\limits_{\alpha\to 0}\frac{w(\alpha)}{\abs*{\log \alpha}}=1$. This implies, $\forall n\geq N_3$,
\begin{equation}
\mathop{\lim\sup}_{\alpha\to 0}\frac{\tau}{\abs*{\log \alpha}}\leq \frac{1}{\frac{D^*(\bm\theta)}{(1+\epsilon)}-v'(n)}.
\end{equation}
Letting $n\to\infty$, we get
\begin{equation}
\mathop{\lim\sup}_{\alpha\to 0}\frac{\tau}{\abs*{\log \alpha}}\leq \frac{(1+\epsilon)}{D^*(\bm\theta)}.
\end{equation}
Now letting $\epsilon\to 0$, we get
\begin{equation}
\mathop{\lim\sup}_{\alpha\to 0}\frac{\tau}{\abs*{\log \alpha}}\leq \frac{1}{D^*(\bm\theta)}.
\end{equation}
\end{proof}
\begin{theorem}[Asymptotic optimality in expectation]
Let $\bm\theta\in\bigcup\limits_{m=1}^M \Theta_m$.. The proposed policy satisfies
\begin{equation}
\mathop{\lim\sup}_{\alpha\to 0}\frac{\mathbb{E}_{\bm\theta}[\tau]}{\abs*{\log \alpha}}\leq \frac{1}{D^*(\bm\theta)}.
\end{equation}
\end{theorem}
\begin{proof}
Let $B_\xi(\bm\theta')$ denote the $\xi$-neighbourhood of $\bm\theta'$, that is, $B_\xi(\bm\theta')=\bre*{\bm\theta'':\norm*{\bm\theta''-\bm\theta'}<\xi}$ for $\xi>0$. Let $\mathcal{I}_{\xi}(n)$ be the event given by
\begin{equation}
\mathcal{I}_{\xi}(n)\coloneqq\bre*{\bm\theta^*(n)\in B_\xi(\bm\theta)}.
\end{equation}
Let $g(\bm\theta)=m$ and $\bm\theta\in\Gamma^{(j)}_m$ for some $j\in [x_m]$. Since $\Gamma^{(j)}_m$ is open in $\text{aff}\para*{\Gamma^{(j)}_m}$, $\exists \xi_0>0$ such that $B_{\xi_0}(\bm\theta)\cap \text{ aff}\para*{\Gamma^{(j)}_m}\subset\Gamma^{(j)}_m$. Let $\xi'_0 = \min\limits_{m'\neq m\text{ or }i\neq j}\Delta (\bm\theta,\Gamma^{(i)}_{m'})>0$. Thus $\forall \xi\in \para*{0,\min(\xi_0,\frac{\xi'_0}{2})}$,
\begin{align}
\mathcal{I}_\xi(n) &\implies \hat{r}(n) =m\\
&\implies \norm*{\hat{\bm\theta}(n)-\bm\theta^*(n)}\leq \rho\Delta (\bm\theta^*(n),\Theta_m)\\
&\implies\norm*{\hat{\bm\theta}(n)-\bm\theta^*(n)}\leq \rho\norm*{\bm\theta-\bm\theta^*(n)}\\
&\implies \norm*{\hat{\bm\theta}(n)-\bm\theta^*(n)}\leq \rho\xi\\
&\implies \norm*{\hat{\bm\theta}(n)-\bm\theta^*(n)}+\norm*{\bm\theta^*(n)-\bm\theta}\leq (1+\rho)\xi\\
&\implies \norm*{\hat{\bm\theta}(n)-\bm\theta}\leq (1+\rho)\xi.\label{eq:I}
\end{align}
From proposition \ref{prop:q} and \eqref{eq:I}, we get that given $\epsilon>0$, $\exists \xi_1(\epsilon)\in\para*{0,\min(\xi_0,\frac{\xi'_0}{2})}$ such that
\begin{equation}\label{eq:q_epsilon}
\mathcal{I}_{\xi_1(\epsilon)}(n)\implies \max\limits_{u\in\mathcal{U}}\abs*{q^*_u(\hat{\bm\theta}(n))-q^*_u(\bm\theta)}<\epsilon.
\end{equation}
Let $N\in\mathbb{N}$ and the event
\begin{equation}
\mathcal{E}_N(\epsilon)=\bigcap_{n=N^{1/4}}^N \mathcal{I}_{\xi_1(\epsilon)}(n).
\end{equation}
The following claim is a consequence of the `forced exploration' by the control law which ensures that each control is chosen at least around $\sqrt{n}$ times at time $n$.
\begin{claim}\label{claim:2} $\exists K,C$ which are constants that depend on $\epsilon$ and $\bm\theta$ such that $\forall N\geq N'=3^4\abs*{\mathcal{U}}^8+1$,
\begin{equation}
\mathbb{P}_{\bm\theta}\bra*{ \mathcal{E}_N^c  (\epsilon)}\leq KN\exp\para*{-CN^{1/8}}.
\end{equation}
\end{claim}
\begin{proof}
Let $N\geq N'$. Thus, $\forall n\in [\sqrt{N},N]\cap\mathbb{N}$, we get $\forall u\in\mathcal{U}, N_u(n)\geq \sqrt{n+\abs*{\mathcal{U}}^2}-2\abs*{\mathcal{U}}>0$. Note that
\begin{align}
\mathbb{P}_{\bm\theta}\bra*{\mathcal{E}_N^c (\epsilon)}&\leq \sum\limits_{n=N^{1/4}}^N \mathbb{P}_{\bm\theta} \bra*{\bm\theta^*(n)\notin B_{\xi_1(\epsilon)}(\bm\theta)}\\
&\leq \sum\limits_{n=N^{1/4}}^N \sum\limits_{u=1}^{\abs*{\mathcal{U}}} \mathbb{P}_{\bm\theta}\bra*{\theta^*_u(n)\notin\para*{\theta_u-\xi,\theta_u+\xi}},
\end{align}
for some $\xi>0$.
Using a union bound and Chernoff inequality, we get that $\forall u\in\mathcal{U}$
\begin{align}
\mathbb{P}_{\bm\theta} [\theta^*_u(n)\leq\theta_u-\xi]&=\mathbb{P}_{\bm\theta} \bra*{\theta^*_u(n)\leq\theta_u-\xi,N_u(n)\geq s(n)=\sqrt{n+\abs*{\mathcal{U}}^2}-2\abs*{\mathcal{U}}}\\
&\leq \sum\limits_{k=s(n)}^n \mathbb{P}_{\bm\theta} \bra*{\theta^*_u(n)\leq\theta_u-\xi,N_u(n)=k}\\
&\leq \sum\limits_{k=s(n)}^n \exp \para*{-kD_u(\bm\theta-\xi||\bm\theta)}\\
&\leq \frac{e^{-s(n)D_u(\bm\theta-\xi||\bm\theta)}}{1-e^{-D_u(\bm\theta-\xi||\bm\theta)}},
\end{align}
where $\bm\theta+x$ is a vector given by $(\theta_1+x,\theta_2+x,\dots,\theta_{\abs*{\mathcal{U}}}+x)$ for any scalar $x$.
Similarly, we get
\begin{equation}
\mathbb{P}_{\bm\theta} \bra*{\theta^*_a(n)\leq\theta_u+\xi}\leq \frac{e^{-s(n)D_u(\bm\theta+\xi||\bm\theta)}}{1-e^{-D_u(\bm\theta+\xi||\bm\theta)}}.
\end{equation}
Let
\begin{equation}
C=\min_{u\in\mathcal{U}} \para*{D_u(\bm\theta-\xi||\bm\theta)\lor D_u(\bm\theta+\xi||\bm\theta)}
\end{equation}
and
\begin{equation}
K=\sum\limits_{u=1}^{\abs*{\mathcal{U}}} \para*{\frac{e^{2\abs*{\mathcal{U}}D_u(\bm\theta-\xi||\bm\theta)}}{1-e^{-D_u(\bm\theta-\xi||\bm\theta)}}+\frac{e^{2\abs*{\mathcal{U}})D_u(\bm\theta+\xi||\bm\theta)}}{1-e^{-D_u(\bm\theta+\xi||\bm\theta)}}}.
\end{equation}
Thus we get,
\begin{align}
\mathbb{P}_{{\bm\theta}}\bra*{\mathcal{E}_N^c (\epsilon)}&\leq \sum\limits_{n=N^{1/4}}^N K\exp\para*{-C\sqrt{n+\abs*{\mathcal{U}}^2}}\\
&\leq KN\exp\para*{-C\sqrt{N^{1/4}+\abs*{\mathcal{U}}^2}}\\
&\leq KN\exp\para*{-C\sqrt{N^{1/4}}}\\
&=KN\exp\para*{-CN^{1/8}}.
\end{align}
\end{proof}
 The next claim discusses the convergence of empirical proportions on $\mathcal{E}_N(\epsilon)$.
\begin{claim}\label{claim:3}
$\exists N_\epsilon$ such that for $N\geq N_\epsilon$, it holds that on $\mathcal{E}_N(\epsilon)$,
\begin{equation}
\forall n\in [\sqrt{N},N]\cap\mathbb{N},\,\, \max_{u\in\mathcal{U}} \abs*{\frac{N_u(n)}{n}-q_u^*(\bm\theta)}\leq 2\epsilon.
\end{equation}
\end{claim}
\begin{proof}
Using lemma \ref{lem:ctrack} and \eqref{eq:q_epsilon}, we get that on $\mathcal{E}_N(\epsilon)$, $\forall n\in[\sqrt{N},N]\cap\mathbb{N}$ and $\forall u\in\mathcal{U}$,
\begin{align}
\abs*{\frac{N_u(n)}{n}-q_u^*(\bm\theta)}&\leq \abs*{\frac{N_u(n)}{n}-\frac{1}{n}\sum\limits_{k=0}^{n-1}q_u^*(\hat{\bm\theta}(k)}+\abs*{\frac{1}{n}\sum\limits_{k=0}^{n-1}q_u^*(\hat{\bm\theta}(k))-q_u^*(\bm\theta)}\\
&\leq \frac{\abs*{\mathcal{U}}(\sqrt{n}+1)}{n}+\frac{N^{1/4}}{n}+\frac{1}{n}\sum\limits_{k=N^{1/4}}^{n-1}\abs*{q_u^*(\hat{\bm\theta}(k))-q_u^*(\bm\theta)}\\
&\leq \frac{2\abs*{\mathcal{U}}}{N^{1/4}}+\frac{1}{N^{1/4}}+\epsilon\\
&=\frac{2\abs*{\mathcal{U}}+1}{N^{1/4}}+\epsilon\\
&\leq 2\epsilon
\end{align}
when $N\geq \para*{\frac{2\abs*{\mathcal{U}}+1}{\epsilon}}^4=N_{\epsilon}$.
\end{proof}
Note that the GLRT statistic can be bounded below as follows.
\begin{align}
Z(n)&=\max_{i\in\mathcal{M}}\min_{j\neq i} Z_{i,j}(n)\\
&\geq \min_{i\neq g(\bm\theta)} Z_{g(\bm\theta),i}(n)\\
&=\min_{i\neq g(\bm\theta)} \log\frac{\sup\limits_{\bm\theta'\in\Theta_{g(\bm\theta)}}\prod\limits_{u=1}^{\abs*{\mathcal{U}}} p(\underbar{$Y$}^u(n);\bm\theta',u)}{\sup\limits_{\bm\theta''\in\Theta_i}\prod\limits_{u=1}^{\abs*{\mathcal{U}}} p(\underbar{$Y$}^u(n);\bm\theta'',u)}\\
&\geq \log\frac{\prod\limits_{u=1}^{\abs*{\mathcal{U}}} p(\underbar{$Y$}^u(n);\bm\theta,u)}{\sup\limits_{\bm\theta''\in\Omega}\prod\limits_{u=1}^{\abs*{\mathcal{U}}} p(\underbar{$Y$}^u(n);\bm\theta'',u)}\\
&= \log\frac{\prod\limits_{u=1}^{\abs*{\mathcal{U}}} p(\underbar{$Y$}^u(n);\bm\theta,u)}{\prod\limits_{u=1}^{\abs*{\mathcal{U}}} p(\underbar{$Y$}^u(n);\bm\theta^*(n),u)}\\
&=n\sum\limits_{u=1}^{\abs*{\mathcal{U}}}\frac{N_u(n)}{n}\bra*{D_u(\bm\theta||\bm\theta^*(n))+(\theta_u-\theta^*_u(n))\para*{\frac{S_u(n)}{N_u(n)}-\dot{A}_u(\theta_u)}}\\
&=n\sum\limits_{u=1}^{\abs*{\mathcal{U}}}\frac{N_u(n)}{n}\bra*{D_u(\bm\theta||\bm\theta^*(n))+(\theta_u-\theta^*_u(n))\para*{\dot{A}_u(\theta^*_u(n))-\dot{A}_u(\theta_u)}}\\
&=np\para*{\bm\theta^*(n),\para*{\frac{N_u(n)}{n}}_{u=1}^{\abs*{\mathcal{U}}}},
\end{align}
where $p:\Omega\times\mathcal{P}\to\mathbb{R}$ is the function given by,
\begin{equation}
p(\bm\theta',\bm{q})\coloneqq \sum\limits_{u=1}^{\abs*{\mathcal{U}}}q_u\bra*{D_u(\bm\theta||\bm\theta')+(\theta_u-\theta'_u)\para*{\dot{A}_u(\theta'_u)-\dot{A}_u(\theta_u)}}.
\end{equation}
Let
\begin{equation}
C_\epsilon^*=\inf_{\substack{\bm\theta':\norm*{\bm\theta'-\bm\theta}\leq \xi_1(\epsilon)\\\bm{q}:\norm*{\bm{q}-\bm{q}^*(\bm\theta)}\leq 2\epsilon}}p(\bm\theta',\bm{q}).
\end{equation}
By the definition of $\mathcal{I}_{\xi_1(\epsilon)}(n)$ and claim \ref{claim:3}, for $N\geq N_\epsilon$, on the event $\mathcal{E}_N(\epsilon)$, it holds that $\forall n\in [\sqrt{N},N]\cap\mathbb{N}$,
\begin{equation}
Z(n)\geq nC^*_\epsilon.
\end{equation}
Let $N\geq N_\epsilon$. On the event $\mathcal{E}_N(\epsilon)$,
\begin{align}
\min(\tau,N)&\leq \sqrt{N}+\sum\limits_{n=\sqrt{N}}^{N} \mathbbm{1}_{\tau>n}\\
&\leq \sqrt{N}+\sum\limits_{n=\sqrt{N}}^N \mathbbm{1}_{Z(n)\leq \beta(n,\alpha)}\\
&\leq \sqrt{N}+\sum\limits_{n=\sqrt{N}}^N \mathbbm{1}_{n C^*_\epsilon\leq \beta(n,\alpha)}\\
&\leq \sqrt{N}+\frac{\beta(N,\alpha)}{C^*_\epsilon}.
\end{align}
We define
\begin{equation}
N_0(\alpha) \coloneqq \inf\bre*{N\in\mathbb{N}:\sqrt{N}+\frac{\beta(N,\alpha)}{C^*_\epsilon}\leq N}.
\end{equation}
So $\forall N\geq \max(N_0(\alpha),N_\epsilon)$, on $\mathcal{E}_N(\epsilon)$, we get
\begin{equation}
\min(\tau,N)\leq N
\end{equation}
which implies
\begin{equation}
\tau\leq N.
\end{equation}
Thus $\forall N\geq \max(N',N_0(\alpha),N_\epsilon)$,
\begin{equation}
\mathcal{E}_N(\epsilon)\subseteq (\tau\leq N)
\end{equation}
and consequently,
\begin{equation}
\mathbb{P}_{\bm\theta}[\tau>N]\leq \mathbb{P}_{\bm\theta}[\mathcal{E}_N^c]\leq KN\exp\para*{-CN^{1/8}}.
\end{equation}
So, we can upper bound the expectation of stopping time as
\begin{equation}\label{eq:expect}
\mathbb{E}_{\bm\theta}[\tau]\leq N_0(\alpha)+N'+N_\epsilon+\sum\limits_{N=1}^{\infty}KN\exp\para*{-CN^{1/8}}.
\end{equation}
We now upper bound $N_0(\alpha)$ as follows.
\begin{equation}
N_0(\alpha)\leq \inf\bre*{N\in\mathbb{N}:\sqrt{N}+\frac{v(N)+w(\alpha)}{C^*_\epsilon}\leq N},
\end{equation}
where $v(n)$ and $w(\alpha)$ are as defined in \eqref{eq:v} and \eqref{eq:w} respectively. Let $v_1(t)=\sqrt{t}+\frac{v(t)}{C^*_\epsilon}$. Note that $\forall t>1, v'_1(t)>0$ and $v_1''(t)<0$. Also, $\lim\limits_{t\to\infty}v_1'(t)=0$ Thus, $\exists N_1\in \mathbb{N}$ such that $\forall n\geq N_1$, $n> v_1(n)$. Also, $\exists N_2\in \mathbb{N}$ such that $\forall n\geq N_2$, $v'_1(n)\in (-\frac{1}{2}, \frac{1}{2})$. Let $N_3=\max\bre*{N_1,N_2}$. Note that $N_3$ is independent of $\alpha$. Thus, $\forall n\geq N_3$,
\begin{equation}
N_0(\alpha) \leq n+ \frac{\frac{w(\alpha)}{C^*_\epsilon}}{1-v'_1(n)}
\end{equation}
Consequently $\forall n\geq N_3$,
\begin{equation}
\mathop{\lim\sup}_{\alpha\to 0}\frac{N_0(\alpha)}{\abs*{\log \alpha}}\leq \frac{1}{C^*_\epsilon(1-v'_1(n))},
\end{equation}
since $\lim\limits_{\alpha\to 0}\frac{w(\alpha)}{\abs*{\log \alpha}}=1$.
Letting $n\to\infty$ we get,
\begin{equation}
\mathop{\lim\sup}_{\alpha\to 0}\frac{N_0(\alpha)}{\abs*{\log \alpha}}\leq \frac{1}{C^*_\epsilon}.
\end{equation}
Using this in the inequality \eqref{eq:expect} as $\alpha\to 0$ we get,
\begin{equation}
\mathop{\lim\sup}_{\alpha\to 0}\frac{\mathbb{E}_{\bm\theta} [\tau]}{\abs*{\log \alpha}}\leq \frac{1}{C^*_\epsilon}.
\end{equation}
By the continuity of $p$, we get
\begin{equation}
\lim_{\epsilon\to 0}C^*_\epsilon = D^*(\bm\theta).
\end{equation}
So letting $\epsilon\to 0$ we get,
\begin{equation}
\mathop{\lim\sup}_{\alpha\to 0}\frac{\mathbb{E}_{\bm\theta} [\tau]}{\abs*{\log \alpha}}\leq \frac{1}{D^*(\bm\theta)}.
\end{equation}
\end{proof}

\section{}\label{appendix:C}
In this section we extend the result in Theorem 2 in \cite{magureanu2014lipschitz}, stated for Bernoulli distributions, to single-parameter exponential family distributions.
\begin{lemma}\label{lem:Z}
Let $a>0$, $L\geq 2$. Let $\bm{Z}\in\mathbb{R}^L$ be a random variable such that $\forall \bm\zeta\in\para*{\mathbb{R}^+}^L$
\begin{equation}
\mathbb{P}\bra*{\bm{Z}\geq\bm\zeta}\leq\exp\para*{-a\sum\limits_{l=1}^L \zeta_l}.
\end{equation}
Then $\forall\delta\geq L/a$,
\begin{equation}
\mathbb{P}\bra*{\sum\limits_{l=1}^LZ_l\geq\delta}\leq\para*{\frac{a\delta e}{L}}^Le^{-a\delta}.
\end{equation}
\end{lemma}

\begin{lemma}\label{lem:E}
For any $u\in\mathcal{U}$, let $1\leq t_u\leq n$. Let $\eta>0$. Let $E$ be the event given by
\begin{equation}
E = \bigcap\limits_{u\in\mathcal{U}} \bre*{t_u\leq N_u(n)\leq (1+\eta)t_u}.
\end{equation}
Then for $\beta\geq (1+\eta)(\abs*{\mathcal{U}}+\log 2)$, we have
\begin{equation}
\mathbb{P}_{\bm\theta}\bra*{\mathbbm{1}_E \sum\limits_{u\in\mathcal{U}}N_u(n)D_u(\bm\theta^*(n)||\bm\theta)\geq\beta} \leq 2\para*{\frac{\beta e}{\abs*{\mathcal{U}}}}^{\abs*{\mathcal{U}}} e^{-\frac{\beta}{(1+\eta)}}.
\end{equation}
\end{lemma}
\begin{proof}
We shall show that $\forall \bm\zeta\in (\mathbb{R}^+)^{\abs*{\mathcal{U}}}$,
\begin{equation}
\mathbb{P}_{\bm\theta}\bra*{\bigcap\limits_{u\in\mathcal{U}}\bre*{\mathbbm{1}_E N_u(n)D_u(\bm\theta^*(n)||\bm\theta)\geq\zeta_u}}\leq 2\exp\bra*{-\frac{\sum\limits_{u\in\mathcal{U}}\zeta_u}{(1+\eta)}}.
\end{equation}
Let $\bm\zeta\in (\mathbb{R}^+)^{\abs*{\mathcal{U}}}$. Let $\bm{c}^{(1)}$ and $\bm{c}^{(2)}$ be such that $\forall u\in\mathcal{U}$, $c_u^{(1)}<c_u^{(2)}$ and
\begin{equation}
t_u(1+\eta)D_u(\bm{c}^{(1)}||\bm\theta)=t_u(1+\eta)D_u(\bm{c}^{(2)}||\bm\theta)=\zeta_u.
\end{equation}
Note that $\forall u\in\mathcal{U}$,
\begin{align}
\mathbbm{1}_E N_u(n)D_u(\bm\theta^*(n)||\bm\theta)\geq\zeta_u &\implies \mathbbm{1}_E t_u(1+\eta)D_u(\bm\theta^*(n)||\bm\theta)\geq\zeta_u\\
&\implies E\bigcap\bre*{\bre*{\theta_u^*(n)\geq c_u^{(2)}}\cup\bre*{\theta_u^*(n)\leq c_u^{(1)}}}.
\end{align}
Now for a fixed $\bm{\lambda}$, let
\begin{equation}
M(n)=\exp\bre*{\sum\limits_{u\in\mathcal{U}} \lambda_u S_u(n)-N_u(n) \kappa_u(\lambda_u)}.
\end{equation}
So $\forall n'\leq n$,
\begin{equation}
M(n')=M(n'-1)\prod\limits_{u\in\mathcal{U}}\exp\bre*{\mathbbm{1}_{\bre*{U_{n'-1}=u}}\bra*{\lambda_u T_u(Y_{n'})-\kappa_u(\lambda_u)}}.
\end{equation}
Since $\mathbbm{1}_{\bre*{U_{n'}=u}}$ is $\mathcal{F}_{n'-1}$-measurable and $Y_{n'}$ is conditionally independent of $\mathcal{F}_{n'-1}$, we get
\begin{equation}
\mathbb{E}_{\bm\theta}\bra*{M(n')|\mathcal{F}_{n'-1}}=M(n'-1).
\end{equation}
Hence, $M(n)$ is a martingale and $\mathbb{E}_{\bm\theta}[M(n)]=1$.
$\forall u\in\mathcal{U}$, set $\lambda^{(1)}_u=c_u^{(1)}-\theta_u<0$ and $\lambda^{(2)}_u=c_u^{(2)}-\theta_u>0$, so that for $i\in\bre*{1,2}$, $\lambda^{(i)}_u \dot{A}_u(c_u^{(i)})-\kappa_u(\lambda^{(i)}_u)=D_u(\bm{c}^{(i)}||\bm\theta)$.  Let
\begin{equation}
M_i(n)=\exp\bre*{\sum\limits_{u\in\mathcal{U}} \lambda^{(i)}_u S_u(n)-N_u(n)\kappa_u(\lambda^{(i)}_u)},
\end{equation}
for $i\in\bre*{1,2}$. Hence,
\begin{align}
&\mathbb{P}_{\bm\theta}\bra*{E\bigcap\limits_{u\in\mathcal{U}}\bre*{\bre*{\theta^*_u(n)\leq c_u^{(1)}}\cup\bre*{\theta^*_u(n)\geq c_u^{(2)}}}}\\
&\leq\mathbb{P}_{\bm\theta}\bra*{E\bigcap\limits_{u\in\mathcal{U}}\bre*{\dot{A}_u(\theta^*_u(n))\leq \dot{A}_u (c_u^{(1)})}}+\mathbb{P}_{\bm\theta}\bra*{E\bigcap\limits_{u\in\mathcal{U}}\bre*{\dot{A}_u(\theta^*_u(n))\geq \dot{A}_u (c_u^{(2)})}}\\
&\leq \sum\limits_{i=1}^2 \mathbb{P}_{\bm\theta}\bra*{\mathbbm{1}_E\sum\limits_{u\in\mathcal{U}}\lambda_u^{(i)}S_u(n)\geq\sum\limits_{u\in\mathcal{U}}N_u(n)\lambda_u^{(i)}\dot{A}_u (c_u^{(i)})}\\
&\leq \sum\limits_{i=1}^2 \mathbb{P}_{\bm\theta}\bra*{\mathbbm{1}_E\sum\limits_{u\in\mathcal{U}}\lambda_u^{(i)}S_u(n)-N_u(n)\kappa_u(\lambda^{(i)}_u)\geq\sum\limits_{u\in\mathcal{U}}N_u(n)\lambda_u^{(i)}\dot{A}_u (c_u^{(i)})-N_u(n)\kappa_u(\lambda^{(i)}_u)}\\
&\leq \sum\limits_{i=1}^2 \mathbb{P}_{\bm\theta}\bra*{\mathbbm{1}_E M_i(n)\geq \exp\para*{\sum\limits_{u\in\mathcal{U}}N_u(n)\para*{\lambda^{(i)}_u \dot{A}_u(c_u^{(i)})-\kappa_u(\lambda^{(i)}_u)}}}\\
&=\sum\limits_{i=1}^2 \mathbb{P}_{\bm\theta}\bra*{\mathbbm{1}_E M_i(n)\geq \exp\para*{\sum\limits_{u\in\mathcal{U}}N_u(n)D_u(\bm{c}^{(i)}||\bm\theta)}}\\
&\leq\sum\limits_{i=1}^2 \mathbb{P}_{\bm\theta}\bra*{\mathbbm{1}_E M_i(n)\geq \exp\para*{\sum\limits_{u\in\mathcal{U}}t_u D_u(\bm{c}^{(i)}||\bm\theta)}}\\
&\leq \sum\limits_{i=1}^2 \frac{\mathbb{E}_{\bm\theta}\bra*{\mathbbm{1}_E M_i(n)}}{\exp\para*{\sum\limits_{u\in\mathcal{U}}t_u D_u(\bm{c}^{(i)}||\bm\theta)}}\\
&\leq \sum\limits_{i=1}^2 \frac{\mathbb{E}_{\bm\theta}\bra*{M_i(n)}}{\exp\para*{\sum\limits_{u\in\mathcal{U}}t_u D_u(\bm{c}^{(i)}||\bm\theta)}}\\
&=\sum\limits_{i=1}^2\exp\para*{-\sum\limits_{u\in\mathcal{U}}t_u D_u(\bm{c}^{(i)}||\bm\theta)}\\
&=2\exp\para*{-\sum\limits_{u\in\mathcal{U}}\frac{\zeta_u}{(1+\eta)}}.
\end{align}
Thus,
\begin{equation}
\mathbb{P}_{\bm\theta}\bra*{\bigcap\limits_{u\in\mathcal{U}}\bre*{\mathbbm{1}_E N_u(n)D_u(\bm\theta^*(n)||\bm\theta)\geq\zeta_u}}\leq 2\exp\para*{-\sum\limits_{u\in\mathcal{U}}\frac{\zeta_u}{(1+\eta)}}.
\end{equation}
Let $Z_u=\mathbbm{1}_E N_u(n)D_u(\bm\theta^*(n)||\bm\theta)$ and $a=\frac{1}{1+\eta}$. Note that we have $\forall \bm\zeta\in (\mathbb{R}^+)^{\abs*{\mathcal{U}}}$,
\begin{equation}
\mathbb{P}_{\bm\theta}\bra*{\bm{Z}\geq\bm\zeta}\leq 2\exp\para*{-a\sum\limits_{u\in\mathcal{U}}\zeta_u}.
\end{equation}
Let $\bm{Z}',\bm\zeta'$ be such that $\forall u\in\mathcal{U}, Z'_u=Z_u-\frac{\log 2}{a\abs*{\mathcal{U}}}$ and $\zeta'_u=\zeta_u-\frac{\log 2}{a\abs*{\mathcal{U}}}$. Thus,
\begin{align}
\mathbb{P}_{\bm\theta}\bra*{\bm{Z}'\geq\bm\zeta'}&\leq 2\exp\bra*{-a\sum\limits_{u\in\mathcal{U}}\para*{\zeta'_u+\frac{\log 2}{a\abs*{\mathcal{U}}}}}\\
&=\exp\para*{-a\sum\limits_{u\in\mathcal{U}}\zeta'_u}.
\end{align}
This holds for all $\bm\zeta'$ such that $\forall u\in\mathcal{U}$, $\zeta'_u\geq -\frac{\log 2}{a\abs*{\mathcal{U}}}$. Hence, applying lemma \ref{lem:Z} we get that $\forall \delta\geq\frac{\abs*{\mathcal{U}}}{a}$,
\begin{equation}
\mathbb{P}_{\bm\theta}\bra*{\sum\limits_{u\in\mathcal{U}} Z'_u\geq\delta}\leq \para*{\frac{a\delta e}{\abs*{\mathcal{U}}}}^{\abs*{\mathcal{U}}}e^{-a\delta}.
\end{equation}
Thus, $\forall \delta\geq\frac{\abs*{\mathcal{U}}+\log 2}{a}$,
\begin{align}
\mathbb{P}_{\bm\theta}\bra*{\sum\limits_{u\in\mathcal{U}} Z_u\geq\delta}&\leq \para*{\frac{a\para*{\delta-\frac{\log 2}{a}}e}{\abs*{\mathcal{U}}}}^{\abs*{\mathcal{U}}}e^{-a\para*{\delta-\frac{\log 2}{a}}}\\
&\leq 2\para*{\frac{a\delta e}{\abs*{\mathcal{U}}}}^{\abs*{\mathcal{U}}}e^{-a\delta}.
\end{align}
Hence, we get that for any $\beta\geq (1+\eta)(\abs*{\mathcal{U}}+\log 2)$,
\begin{equation}
\mathbb{P}_{\bm\theta}\bra*{\mathbbm{1}_E \sum\limits_{u\in\mathcal{U}}N_u(n)D_u(\bm\theta^*(n)||\bm\theta)\geq\beta} \leq 2\para*{\frac{\beta e}{\abs*{\mathcal{U}}}}^{\abs*{\mathcal{U}}} e^{-\frac{\beta}{(1+\eta)}}.
\end{equation}
\end{proof}

\begin{theorem} \label{thm:conc}
\begin{equation}
\mathbb{P}_{\bm\theta}\bra*{\sum\limits_{u\in\mathcal{U}}N_u(n)D_u(\bm\theta^*(n)||\bm\theta)\geq\beta}\leq 2e^{-\beta}\para*{\frac{\beta\ceil*{\beta\log n}}{\abs*{\mathcal{U}}}}^{\abs*{\mathcal{U}}} e^{\abs*{\mathcal{U}}+1},
\end{equation}
for $\beta\geq \abs*{\mathcal{U}}+1+\log 2$.
\end{theorem}
\begin{proof}
Let $\beta\geq \abs*{\mathcal{U}}+1+\log 2$ and $\eta=\frac{1}{\beta-1}$. Let $L=\ceil*{\frac{\log n}{\log (1+\eta)}}$. Let $\mathbb{L}=\bre*{1,2,\dots,L}^{\abs*{\mathcal{U}}}$. Let $G$ be the event
\begin{equation}
G = \bre*{\sum\limits_{u\in\mathcal{U}}N_u(n)D_u(\bm\theta^*(n)||\bm\theta)\geq\beta}.
\end{equation}
Let $H_{\bm{l}}$ be the event
\begin{equation}
H_l=\bigcap\limits_{u\in\mathcal{U}}\bre*{(1+\eta)^{l_u-1}\leq N_u(n) \leq (1+\eta)^{l_u}},
\end{equation}
for any $\bm{l}\in\mathbb{L}$. We have
\begin{equation}
G =\bigcup\limits_{\bm{l}\in\mathbb{L}} G\cap H_{\bm{l}}.
\end{equation}
Thus,
\begin{equation}
\mathbb{P}_{\bm\theta}[G]\leq \sum\limits_{\bm{l}\in\mathbb{L}}\mathbb{P}_{\bm\theta}[G\cap H_{\bm{l}}].
\end{equation}
Note that since $\beta\geq \abs*{\mathcal{U}}+1+\log 2$ and $\eta=\frac{1}{\beta-1}$, we get $\beta\geq (1+\eta)(\abs*{\mathcal{U}}+\log 2)$. So, applying lemma \ref{lem:E}, we get that $\forall \bm{l}\in\mathbb{L}$,
\begin{equation}
\mathbb{P}_{\bm\theta}[G\cap H_{\bm{l}}]\leq 2\para*{\frac{\beta e}{\abs*{\mathcal{U}}}}^{\abs*{\mathcal{U}}} e^{-\frac{\beta}{(1+\eta)}}.
\end{equation}
Since $\abs*{\mathbb{L}}=L^{\abs*{\mathcal{U}}}$,
\begin{equation}
\mathbb{P}_{\bm\theta}[G]\leq 2\para*{\frac{L\beta e}{\abs*{\mathcal{U}}}}^{\abs*{\mathcal{U}}} e^{-\frac{\beta}{(1+\eta)}}.
\end{equation}
Note that $\log (1+\eta)\geq 1-\frac{1}{1+\eta}=\frac{1}{\beta}$. Hence,
\begin{equation}
\mathbb{P}_{\bm\theta}[G]\leq 2e^{-\beta}\para*{\frac{\beta\ceil*{\beta\log n}}{\abs*{\mathcal{U}}}}^{\abs*{\mathcal{U}}} e^{\abs*{\mathcal{U}}+1}.
\end{equation}
\end{proof}
\end{appendices}
\bibliographystyle{IEEEtran}
\bibliography{refs}
\end{document}